\documentclass[10pt,oneside,english]{amsart}
\usepackage{mathpazo}
\usepackage[T1]{fontenc}
\usepackage[latin9]{inputenc}
\usepackage{geometry}
\usepackage{units}
\usepackage{enumitem}
\usepackage{amstext}
\usepackage{amsthm}
\usepackage{amssymb}
\usepackage{graphicx}

\makeatletter

\let\SF@@footnote\footnote
\def\footnote{\ifx\protect\@typeset@protect
    \expandafter\SF@@footnote
  \else
    \expandafter\SF@gobble@opt
  \fi
}
\expandafter\def\csname SF@gobble@opt \endcsname{\@ifnextchar[
  \SF@gobble@twobracket
  \@gobble
}
\edef\SF@gobble@opt{\noexpand\protect
  \expandafter\noexpand\csname SF@gobble@opt \endcsname}
\def\SF@gobble@twobracket[#1]#2{}

\numberwithin{equation}{section}
\numberwithin{figure}{section}
\theoremstyle{plain}
\newtheorem{thm}{\protect\theoremname}[section]
\theoremstyle{plain}
\newtheorem{rem}[thm]{\protect\remarkname}
\theoremstyle{plain}
\newtheorem{lem}[thm]{\protect\lemmaname}
\theoremstyle{definition}
\newtheorem{defn}[thm]{\protect\definitionname}
\theoremstyle{plain}
\newtheorem{prop}[thm]{\protect\propositionname}
\theoremstyle{plain}
\newtheorem{cor}[thm]{\protect\corollaryname}
\theoremstyle{remark}
\newtheorem*{claim*}{\protect\claimname}

\makeatother

\usepackage{babel}
\providecommand{\claimname}{Claim}
\providecommand{\corollaryname}{Corollary}
\providecommand{\definitionname}{Definition}
\providecommand{\lemmaname}{Lemma}
\providecommand{\propositionname}{Proposition}
\providecommand{\remarkname}{Remark}
\providecommand{\theoremname}{Theorem}

\begin{document}
\title{Tilings in graphons}
\author{Jan Hladký}
\address{Institute of Mathematics, Czech Academy of Science, Žitná 25, 110~00,
Praha, Czechia. The Institute of Mathematics of the Czech Academy
of Sciences is supported by RVO:67985840.}
\thanks{\emph{Hladký:} The research leading to these results has received
funding from the People Programme (Marie Curie Actions) of the European
Union's Seventh Framework Programme (FP7/2007-2013) under REA grant
agreement number 628974.}
\email{hladky@math.cas.cz}
\author{Ping Hu}
\address{School of Mathematics, Sun Yat-sen University, Guangzhou, 510275,
China. This work has been carried out while at: Department of Computer
Science, University of Warwick, Coventry, CV4 7AL, United Kingdom.}
\thanks{\emph{Hu:} This work has received funding from the Leverhulme Trust
2014 Philip Leverhulme Prize of Daniel Kral and from the European
Research Council (ERC) under the European Union\textquoteright s Horizon
2020 research and innovation programme (grant agreement No 648509).}
\email{huping9@mail.sysu.edu.cn}
\author{Diana Piguet}
\address{Institute of Computer Science, Czech Academy of Sciences, Pod Vodárenskou
v\v{e}ží 2, 182~07 Prague, Czech Republic. With institutional support
RVO:67985807.}
\thanks{\emph{Piguet} was supported by the Czech Science Foundation, grant
number GJ16-07822Y}
\email{piguet@cs.cas.cz}
\begin{abstract}
We introduce a counterpart to the notion of tilings, that is vertex-disjoint
copies of a fixed graph $F$, to the setting of graphons. The case
$F=K_{2}$ gives the notion of matchings in graphons. We give a transference
statement that allows us to switch between the finite and limit notion,
and derive several favorable properties, including the LP-duality
counterpart to the classical relation between the fractional vertex
covers and fractional matchings/tilings, and discuss connections with
property testing.

As an application of our theory, we determine the asymptotically almost
sure $F$-tiling number of inhomogeneous random graphs $\mathbb{G}(n,W)$.
As another application, in an accompanying paper {[}Hladký, Hu, Piguet:
Komlós's tiling theorem via graphon covers, J. Graph Theory, 2019{]}
we give a proof of a strengthening of a theorem of Komlós {[}Komlós:
Tiling Turán Theorems, Combinatorica, 2000{]}.
\end{abstract}

\maketitle
\global\long\def\JUSTIFY#1{\mbox{\fbox{\tiny#1}}\quad}%

\global\long\def\SUPPORT{\mathrm{supp\:}}%

\global\long\def\SUPPORTPOSITIVE{\mathrm{supp}}%

\global\long\def\ESSINF{\mathrm{essinf}}%
\global\long\def\ESSSUP{\mathrm{esssup}}%

\global\long\def\TIL{\mathsf{til}}%

\global\long\def\FTIL{\mathsf{ftil}}%

\global\long\def\FCOV{\mathsf{fcov}}%

\global\long\def\DIST{\mathrm{dist}}%

\global\long\def\EXP{\mathbf{E}}%

\global\long\def\diffr{\:\mathsf{d}}%

\global\long\def\WEAKCONV{\overset{\mathrm{w}^{*}}{\;\longrightarrow\;}}%

\global\long\def\CUTNORMCONV{\overset{\|\cdot\|_{\square}}{\;\longrightarrow\;}}%

\section{Towards limits of tilings\label{sec:Intro}\protect\footnote{An extended abstract describing also the main results of this paper
appeared in the proceedings of Eurocomb~2017, \cite{DoHlHuPi:CombinatorialOptimization}.}}

The emergence of graph limit theories has brought numerous novel views
on classical problems in graph theory. More precisely, the problems
in which these theories helped concern comparing subgraph densities.
This is a very broad area lying in the heart of extremal graph theory.
In this explanatory section, we focus only on the applications of
dense graphs limits, firstly because these have been richer, and secondly
because this is the direction which we pursue in this paper. Two closely
related theories have emerged. Razborov's flag algebras method~\cite{Razborov2007}
represents an abstract approach to graph limits. The method gives
universal methods for calculations with subgraph densities (most notably
the semidefinite method and differential calculus), and has led to
the complete solutions of, or at least to a breakthrough progress
on, several prominent problems in the area. Of these breakthroughs
let us mention a result of Razborov~\cite{Razborov:Triangle} who
determined an optimal function $f:[0,1]\rightarrow[0,1]$ such that
if an $n$-vertex graph $G$ contains at least $\alpha{n \choose 2}$
edges then 
\begin{equation}
\mbox{\ensuremath{G} contains at least \ensuremath{\left(f(\alpha)+o(1)\right){n \choose 3}} triangles},\label{eq:RazborovLovaszSimonovits}
\end{equation}
thus resolving an old question of Lovász and Simonovits. On the other
hand, the theory developed by Borgs, Chayes, Lovász, Szegedy, Sós
and Vesztergombi~\cite{Borgs2008c,Lovasz2006} provides explicit
limit objects, the so-called graphons. This theory has been successfully
applied in various parts of graph theory (and in particular provided
insights into the properties of Szemerédi regularity partitions) and
random graphs. In extremal graph theory, the theory of graphons has
been used to prove a certain ``local'' version of Sidorenko's conjecture,~\cite{Lov:Sidorenko}.

As said, the graph limit theories have been very powerful in relating
subgraph densities. Some other concepts, like the one of the minimum
degree have been translated to the setting of graph limits and have
been explored. Extremal graph theory, however, is a much richer field,
and which other statements or features can be formulated in the language
of graph limits is interesting in its own right. In the present paper
we develop a theory of tiling in graphons. A \emph{tiling} $\mathcal{T}$
by a finite graph $F$ in another graph $G$ (``$F$-tiling in $G$'',
in short) is a collection of vertex-disjoint copies of $F$ (not necessarily
induced) in $G$. This concept is used also under the name of $F$-matching,
as we get the usual notion of matchings, when $F=K_{2}$. The \emph{size}
of the tiling $\mathcal{T}$ is simply the cardinality $|\mathcal{T}|$.
We write $\TIL(F,G)$ for the size of a maximum $F$-tiling in $G$.
For example, a ``tiling counterpart'' to~(\ref{eq:RazborovLovaszSimonovits})
would entail finding an optimal function $g:[0,1]\rightarrow[0,1]$
such that any $n$-vertex graph $G$ containing at least $\alpha{n \choose 2}$
edges satisfies 
\[
\mbox{\ensuremath{G} contains a tiling of \ensuremath{K_{3}} of size at least \ensuremath{\left(g(\alpha)+o(1)\right)\frac{n}{3}}}.
\]
Such a function $g$ was indeed determined in~\cite{ABHP:DensityCorHaj}.
Another example is the basic result of Erd\H{o}s and Gallai~\cite{Erdos1959},
where they determined the size of a matching guaranteed in a graph
of a given density.

There is a fractional relaxation of the notion of tilings. In that
notion, we put $[0,1]$-weights on copies of $F$ in $G$, and require
that the total weight at each vertex in $G$ is at most~1. However,
for us it is more convenient to choose a slightly different notion
in which we replace copies by homomorphic copies. More precisely,
we assume that the graph $F$ is on the vertex set $[k]$. We denote
by $\mathcal{F}_{F}(G)$ all the copies of $F$ in the graph $G$,
\begin{equation}
\mathcal{F}_{F}(G)=\left\{ \left(u_{1},u_{2},\ldots,u_{k}\right)\in V(G)^{k}:u_{i}u_{j}\in E(G)\mbox{ for each pair }ij\in E(F)\right\} \;.\label{eq:FFfinite}
\end{equation}
More precisely, the members of $\mathcal{F}_{F}(G)$ correspond to
ordered vertex-sets of $G$ that represent homomorphisms of $F$ into
$G$. With a slight abuse of notation we shall call members of $\mathcal{F}_{F}(G)$
\emph{copies of $F$ in $G$}. Given a copy $F'$ of $F$ in $G$,
the \emph{vertex set} $V(F')$ is defined in an obvious way. Note
that $|V(F')|\le k$, but equality need not hold.

A \emph{fractional $F$-tiling} in a finite graph $G$ is a weight
function $\mathfrak{t}:\mathcal{F}_{F}(G)\rightarrow[0,1]$ that satisfies
that for each $v\in V(G)$, 
\[
\ensuremath{\sum_{F'\in\mathcal{F}_{F}(G),V(F')\ni v}\mathfrak{t}(F')\le1}\;.
\]
The \emph{size }$\|\mathfrak{t}\|$ of \emph{$\mathfrak{t}$} is the
total weight of $\mathcal{F}_{F}(G)$, i.e., $\|\mathfrak{t}\|=\sum_{F'\in\mathcal{F}_{F}(G)}\mathfrak{t}(F')$.
A standard compactness argument shows, that there exists an $F$-tiling
of maximum size. We denote this maximum size by $\FTIL(F,G)$ and
call it the \emph{fractional $F$-tiling number}. Each $F$-tiling
$\mathcal{T}\subset\mathcal{F}_{F}(G)$ can be represented as a fractional
$F$-tiling by simply putting weight~1 on the copies of $\mathcal{T}$
and~0 on the copies of $\mathcal{F}_{F}(G)\setminus\mathcal{T}$.
Thus we have $\TIL(F,G)\le\FTIL(F,G)$.
\begin{rem}
To illustrate the notion, and in particular to emphasize the difference
between copies and homomorphic copies, let us compute $\FTIL(F,G)$
when $F$ is a five-cycle and $G$ is a triangle (so, $G$ is the
smaller of the graphs!). Clearly, $\FTIL(C_{5},K_{3})\le\frac{3}{5}$,
which is just a particular instance of the general bound $\FTIL(F,G)\le\frac{v(G)}{v(F)}$.
On the other hand, $1\mapsto1$, $2\mapsto2$, $3\mapsto1$, $4\mapsto2$,
$5\mapsto3$ is a homomorphism of $C_{5}$ to $K_{3}$. Now we can
consider 2 further variants of this homomorphism obtained by cyclic
shifts of $K_{3}$, and putting weight $\frac{1}{5}$ on each of these~3
homomorphic copies. This shows that~$\FTIL(C_{5},K_{3})\ge\frac{3}{5}$.
\end{rem}

The linear programming (LP) duality provides a very useful way of
expressing $\FTIL(F,G)$. To this end recall that a function $\mathfrak{c}:V(G)\rightarrow[0,1]$
is a \emph{fractional $F$-cover} of $G$ if $\sum_{v\in F'}\mathfrak{c}(v)\ge1$
for each copy $F'\in\mathcal{F}_{F}(G)$ (so, in the summation, we
view $F'$ as a multiset). The \emph{size} of the fractional cover
$\mathfrak{c}$ is the total weight of $V(G)$, and is denoted by
$\|\mathfrak{c}\|$. Again, a compactness argument shows that there
exists a fractional $F$-cover of $G$ of minimum size, which is denoted
by $\FCOV(F,G)$ and called the \emph{fractional $F$-cover number}.
Then the Duality Theorem asserts that $\FCOV(F,G)=\FTIL(F,G).$ Let
us note that showing the ``$\ge$'' direction is easy and the difficulty
lies in proving the ``$\le$'' direction.

\medskip{}

In this paper we define the notions of (fractional) $F$-tilings and
fractional $F$-covers for graphons (Definition~\ref{def:graphontiling}
and Definition~\ref{def:covergraphon}). As we show, when the graphon
in question is taken to be a representation of a finite graph then
there is a correspondence between the classical finite notion and
the new graphon notion (see Propositions~\ref{prop:FTILgraphTILgraphon}
and \ref{prop:FCOVgraphFCOV}). We relate these graphon notions to
the limits of the corresponding finite parameters (see Theorem~\ref{thm:lowersemicontgraphs},
Proposition~\ref{prop:uppercont}, Corrolary~\ref{cor:limesinferior})
and treat their continuity properties on the graphon space (see Theorem~\ref{thm:lowersemicont},
Theorem~\ref{thm:lowersemicont}, and Theorem~\ref{thm:limitOfCovers}).
We derive the LP-duality between the fractional $F$-tiling number
and fractional $F$-cover number for graphons (Theorem~\ref{thm:LPdualityGraphons}).

This paper is organized as follows. In Section~\ref{sec:Preliminaries}
we introduce common notation and provide preliminaries regarding functional
analysis, graph limits and the regularity lemma. In Section~\ref{sec:StatementOfTheResults}
we introduce all the main concepts in our theory and state the main
results. In Section~\ref{sec:Proofs} we give proofs of these results.
In Section~\ref{sec:RandomGraphs} we determine the $F$-tiling number
of inhomogeneous random graphs. In an accompanying paper~\cite{HlHuPi:Komlos}
we give another application of the theory by proving a strengthened
version of a theorem of Komlós~\cite{Komlos2000} regarding tilings
in finite graphs.

Let us also mention that Doležal and Hladký~\cite{DoHl:Polytons}
study in more detail the case $F=K_{2}$, that is, the case of matchings
in graphons.

\section{Notation and preliminaries\label{sec:Preliminaries}}

Suppose that $X$ is an arbitrary set. Given $\mathbf{x}=(x_{1},\ldots,x_{k})\in X^{k}$
and two numbers $1\le i<j\le k$, we write $\pi_{ij}(\mathbf{x})$
for the projection of $\mathbf{x}$, $\pi_{ij}(\mathbf{x})=(x_{i},x_{j})$.
We extend this to projecting sets, i.e., given $Y\subset X^{k}$,
we write $\pi_{ij}(Y)=\left\{ \pi_{ij}(\mathbf{x}):\mathbf{x}\in Y\right\} $.
We will use this notion only with regard to the inverse map. That
is, if $A\subset X^{2}$ then $\pi_{ij}^{-1}(A)=\left\{ \mathbf{x}\in X^{k}:(x_{i},x_{j})\in A\right\} $.

Given a function $f$ and a number $a$ we define its \emph{support}
$\SUPPORT f=\{x:f(x)\neq0\}$ and its variant $\SUPPORTPOSITIVE_{a}\:f=\{x:f(x)\ge a\}$.

Our notation follows~\cite{Lovasz2012}. Throughout the paper we
shall assume that $\Omega$ is an atomless Borel probability space
equipped with a measure $\nu$ (defined on an implicit $\sigma$-algebra).
The product measure on $\Omega^{k}$ is denoted by $\nu^{k}$. Recall
that a set is \emph{null} if it has zero measure.

\subsection{Banach\textendash Alaoglu Theorem}

As $\Omega$ is a Borel probability space, it is in particular a separable
measure space. It is well-known that then the Banach space $\mathcal{L}^{1}(\Omega)$
is separable (see e.g.~\cite[Theorem 13.8]{BrBrTh:RealAnalysis}).
The dual of $\mathcal{L}^{1}(\Omega)$ is $\mathcal{L}^{\infty}(\Omega)$. 

Let us now recall the defining property of the weak$^{*}$ topology
on $\mathcal{L}^{\infty}(\Omega)$: a sequence $f_{1},f_{2},\ldots\in\mathcal{L}^{\infty}(\Omega)$
converges weak$^{*}$ to a function $f\in\mathcal{L}^{\infty}(\Omega)$
if for each $g\in\mathcal{L}^{1}(\Omega)$ we have that $\int_{x\in\Omega}f_{n}(x)g(x)\diffr\nu\rightarrow\int_{x\in\Omega}f(x)g(x)\diffr\nu$.
Weak{*} convergence is denoted by $f_{n}\WEAKCONV f.$

These initial preparations can be used to verify that in the current
context the assumptions of the sequential Banach\textendash Alaoglu
Theorem (as stated for example in~\cite[Theorem 1.9.14]{Tao:EpsilonRoomI})
are fulfilled. Thus, let us state the theorem in the setting of $\mathcal{L}^{\infty}(\Omega)$.
\begin{thm}
\label{thm:BanachAlauglou}If $\Omega$ is a Borel probability space
then each sequence of functions of $\mathcal{L}^{\infty}(\Omega)$-norm
at most~1 contains a weak{*} convergent subsequence.
\end{thm}

\subsection{Graphons}

Our graphons will be defined on $\Omega^{2}$, i.e., symmetric measurable
functions $W:\Omega^{2}\rightarrow[0,1]$. For a finite graph $G$
on vertex set $\{v_{1},\ldots,v_{n}\}$ we can partition $\Omega$
into $n$ sets $\Omega_{1},\ldots,\Omega_{n}$ of measure $\frac{1}{n}$
each and obtain a \emph{representation }of $G$. This graphon, always
denoted by~$W_{G}$, is defined to be one or zero on each square
$\Omega_{i}\times\Omega_{j}$ depending on whether the pair $v_{i}v_{j}$
does or does not form an edge, respectively. Note that~$W_{G}$ is
not uniquely defined; it depends on the partition $\Omega_{1},\ldots,\Omega_{n}$.
Given a bounded symmetric measurable function $W:\Omega^{2}\rightarrow\mathbb{R}$,
define its \emph{cut norm} $\|W\|_{\square}$ to be $\textrm{sup}_{S,T\subseteq\Omega}\left|\int_{(x,y)\in S\times T}W(x,y)\diffr\nu^{2}\right|$,
where the supremum is taken over all measurable subsets $S$ and $T$.
Let $W$ be a graphon and let $\varphi:\Omega\rightarrow\Omega$ be
a measure preserving map. Define $W^{\varphi}$ by $W^{\varphi}(x,y)=W(\varphi(x),\varphi(y)).$
Given two graphons $U$ and $W$, we define the \emph{cut distance}
of them by $\delta_{\square}(U,W)=\textrm{\textrm{inf}}_{\varphi}$
$\|U-W^{\varphi}\|_{\square}$, where the infimum is taken over all
measure preserving map $\varphi:\Omega\rightarrow\Omega$. We denote
$U\le W$ $\nu^{2}$-almost everywhere by $U\le W$.

Suppose that we are given an arbitrary graphon $W:\Omega^{2}\rightarrow[0,1]$
and a graph $F$ whose vertex set is $[k]$. We write $W^{\otimes F}:\Omega^{k}\rightarrow[0,1]$
for a function defined by $W^{\otimes F}(x_{1},\ldots,x_{k})=\prod_{1\le i<j\le k,ij\in E(F)}W(x_{i},x_{j})$,
which we call the \emph{density} of $F$ in $W$. Let $\text{\ensuremath{\mathcal{F}}}_{F}(W)=\SUPPORT W^{\otimes F}$.
Note that if we take~$W$ to be a graphon representation a finite
graph $G$ then there is a natural correspondence between $\text{\ensuremath{\mathcal{F}}}_{F}(W)$
and $\text{\ensuremath{\mathcal{F}}}_{F}(G)$ defined in~(\ref{eq:FFfinite}).

We shall need the following easy statement about graphons, which is
essentially given in~\cite[Lemma 4.1]{Lovasz2006}.
\begin{lem}
\label{lem:cutdistVSsubgraphCount}Suppose that $U$ and $W$ are
two graphons such that $\|W-U\|_{\square}<\delta$, and let $F$ be
a finite graph with $k$ vertices. Then for arbitrary measurable sets
$P_{1},\ldots,P_{k}\subset\Omega$ we have 
\[
\left|\int_{(x_{1},\ldots,x_{k})\in\prod_{i}P_{i}}U^{\otimes F}(x_{1},\ldots,x_{k})\diffr\nu^{k}-\int_{(x_{1},\ldots,x_{k})\in\prod_{i}P_{i}}W^{\otimes F}(x_{1},\ldots,x_{k})\diffr\nu^{k}\right|\le\binom{k}{2}\delta.
\]
\qed
\end{lem}

\subsection{Removal lemma}

Let us state a graphon version of the Removal lemma.
\begin{lem}
\label{lem:removal}Suppose that $F$ is a graph on a vertex set $[k]$.
Then for every $\epsilon>0$, there exists a constant $\delta>0$
such that whenever $W:\Omega^{2}\rightarrow[0,1]$ is a graphon with
$\int_{(x_{1},\ldots,x_{k})\in\Omega^{k}}W^{\otimes F}\diffr\nu^{k}<\delta$,
then $W$ can be made $F$-free by decreasing it in a suitable way
such that it changes by at most $\epsilon$ in the $\mathcal{L}^{1}(\Omega^{2})$-distance.
\end{lem}

\subsection{Partite versions of graphons\label{subsec:PartiteGraphons}}

In this section we introduce an auxiliary notion of partite graphons
which we shall use in Section~\ref{subsec:ProofPropUpperCont}. If
$A$ and $B$ are (not necessarily disjoint) sets then we write $A\amalg B$
for their disjoint union, i.e., for the set $\left(\{1\}\times A\right)\cup\left(\{2\}\times B\right)$.
If $A$ and $B$ are measurable subsets of a measure space $(\Omega,\nu)$
then we can view $A$ as a measure space of total measure $\nu(A)$
and $B$ as a measure space of total measure $\nu(B)$. Then, $A\amalg B$
can be viewed as a measure space of total measure $\nu(A)+\nu(B)$.

Suppose that $W:\Omega^{2}\rightarrow[0,1]$ is a graphon and $A_{1},\ldots,A_{k}$
are (not necessarily disjoint) measurable subsets of $\Omega$. Then
we can construct a \emph{$\left(A_{1},\ldots,A_{k}\right)$-partite
version of $W$}, which is a function $U:\left(A_{1}\amalg\ldots\amalg A_{k}\right)^{2}\rightarrow[0,1]$
defined by the formula
\[
U_{\restriction\left(\{i\}\times A_{i}\right)\times\left(\{j\}\times A_{j}\right)}=W_{\restriction A_{i}\times A_{j}},
\]
for each $i,j\in[k]$ distinct, and by setting $U_{\restriction\left(\{i\}\times A_{i}\right)\times\left(\{i\}\times A_{i}\right)}$
to~0 for each $i\in[k]$. $U$ is almost a graphon; the only reason
why it need not be is that the measure on $A_{1}\amalg\ldots\amalg A_{k}$
defined above, and henceforth denoted by $\nu_{A_{1}\amalg\ldots\amalg A_{k}}$,
need not have total measure~1. So, we shall use graphon terminology
even for partite versions.

If $X\subset\left(A_{1}\amalg\ldots\amalg A_{k}\right)^{2}$ then
the \emph{folded version of $X$} is simply the set of all $(x,y)\in\Omega$
such that there exist some $i,j\in[k]$ such that $\left(\left(i,x\right),(i,y)\right)\in X$.
Note that 
\begin{gather}
\mbox{the \ensuremath{\Omega^{2}}-measure of a folded set is at most }\nonumber \\
\mbox{the \ensuremath{\left(A_{1}\amalg\ldots\amalg A_{k}\right)^{2}}-measure of the original set}\label{eq:projected}
\end{gather}

When $k$ is fixed, the total measure of $\nu_{A_{1}\amalg\ldots\amalg A_{k}}$
is upper-bounded (by $k$), and thus it is easy to translate Lemma~\ref{lem:removal}
as follows.
\begin{lem}
\label{lem:removalpartite}Suppose that $F$ is a graph on a vertex
set $[k]$. Then for every $\epsilon>0$ there exists a constant $\delta>0$
such that the following holds. If $W:\Omega^{2}\rightarrow[0,1]$
is a graphon and $A_{1},\ldots,A_{k}$ are measurable subsets of $\Omega$
with the property that for the $\left(A_{1},\ldots,A_{k}\right)$-version
of $W$, denoted by $U$, we have $\int_{(x_{1},\ldots,x_{k})\in\left(A_{1}\amalg\ldots\amalg A_{k}\right)^{k}}U^{\otimes F}\diffr\nu_{A_{1}\amalg\ldots\amalg A_{k}}^{k}<\delta$
then $U$ can be made $F$-free by decreasing it in a suitable way
such that it changes by at most $\epsilon$ in the $\mathcal{L}^{1}\left(\left(A_{1}\amalg\ldots\amalg A_{k}\right)^{2}\right)$-distance.\qed
\end{lem}

Let us make one easy and well-known comment about the Removal lemma
(this comment applies equally to Lemma~\ref{lem:removal} and~\ref{lem:removalpartite}).
When making a graphon $F$-free with a small change in the $\mathcal{L}^{1}$-distance
it only makes sense to nullify the graphon at each point where its
value is changed. In particular, when positive values of the graphon
are uniformly lower bounded, the nullification occurs on a set of
small measure. This is stated in the lemma below.
\begin{lem}
\label{lem:removalpartiteNullify}Suppose that $F$ is a graph on
a vertex set $[k]$. Then for every $\alpha>0$ and $d>0$ there exists
a constant $\beta>0$ such that the following holds. If $W:\Omega^{2}\rightarrow\{0\}\cup[d,1]$
is a graphon and $A_{1},\ldots,A_{k}$ are measurable subsets of $\Omega$
with the property that for the $\left(A_{1},\ldots,A_{k}\right)$-version
of $W$, denoted by $U$, we have $\int_{(x_{1},\ldots,x_{k})\in\left(A_{1}\amalg\ldots\amalg A_{k}\right)^{k}}U^{\otimes F}\diffr\nu_{A_{1}\amalg\ldots\amalg A_{k}}^{k}<\beta$
then $U$ can be made $F$-free by nullifying it on a suitable set
of $\nu_{A_{1}\amalg\ldots\amalg A_{k}}^{2}$-measure at most $\alpha$.\qed
\end{lem}

\subsection{Regularity lemma}

Here, we introduce Szemerédi's Regularity Lemma in a form that is
convenient for our later purposes. Recall that the \emph{density}
of a bipartite graph with colour classes~$A$ and~$B$ is defined
as $d(A,B)=\frac{e(A,B)}{|A||B|}$. We say that $(A,B)$ forms an
\emph{$\epsilon$-regular pair} if $|d(A,B)-d(X,Y)|<\epsilon$ for
each $X\subset A$ and $Y\subset B$ with $|X|\ge\epsilon|A|$ and
$|Y|\ge\epsilon|B|$. Suppose that~$G$ is a graph on a vertex set
$V$. Let $\mathcal{V}=\left\{ V_{0},V_{1},\ldots,V_{\ell}\right\} $
be a partition of $V$. We say that a subgraph $H$ of $G$ is an
\emph{$(\epsilon,d)$-regularization} of $G$ if
\begin{itemize}
\item $V(H)=V$,
\item $|V_{0}|<\epsilon n$,
\item for each $1\le i<j\le\ell$, we have that $|V_{i}|=|V_{j}|$,
\item for each $1\le i<j\le\ell$, we have that the density of the pair
$(V_{i},V_{j})$ in the graph $H$ is either~0 or at least $d$,
and the pair is $\epsilon$-regular in $H$,
\item there are no edges in $H$ incident to any vertex of $V_{0}$, 
\item there are no edges in the graphs $H[V_{1}],H[V_{2}],\ldots,H[V_{\ell}]$,
and
\item $e(H)\ge e(G)-2d|V|^{2}$.
\end{itemize}
In the setting above, the sets $V_{1},\ldots,V_{\ell}$ are called
\emph{clusters}, and the number $\ell$ is called the \emph{complexity}
of the regularization. The Szemerédi Regularity Lemma then can be
stated as follows.
\begin{lem}
\label{lem:RegularityLemma}For each $d,\epsilon>0$ there exists
a number $L$ such that any graph admits an $(\epsilon,d)$-regularization
of complexity at most $L$.
\end{lem}

Suppose that $G$ is a graph and $H$ is a $(\epsilon,d)$-regularization
of $G$ for a partition $\mathcal{V}=\left\{ V_{0},V_{1},\ldots,V_{\ell}\right\} $.
We can then create the \emph{cluster graph} $R$ corresponding to
the pair $(H,\mathcal{\mathcal{V}})$. This is an edge-weighted graph
on the vertex set $[\ell]$ in which an edge $ij$ is present if the
bipartite graph $H[V_{i},V_{j}]$ has positive density. The weight
on the edge $ij$ in $R$ is then the density of $H[V_{i},V_{j}]$.
A standard and well-known fact states that if $H$ is an $(\epsilon,d)$-regularization
of a graph~$G$, and~$R$ is the corresponding cluster graph, then
for the representations $W_{G}$ and $W_{R}$ we have that 
\begin{equation}
\DIST_{\square}(W_{G},W_{R})\le4d+2\epsilon\;.\label{eq:distGraphClustergraph}
\end{equation}

The so-called Slicing Lemma tells us that a subpair $(X,Y)$ of an
$\epsilon$-regular pair $(A,B)$ is $\left(\nicefrac{\epsilon}{x}\right)$-regular,
if $|X|\ge x|A|$ and $|Y|\ge x|B|$. This fact, however, is useless
when we take~$x$ smaller than~$\epsilon$. The following lemma
shows that we can still save the situation if we take~$X$ and~$Y$
at random. The lemma is a weak version of the the main result of \cite{GKRS:InheritRegularity}.
There, the statement (\cite[Theorem 3.6]{GKRS:InheritRegularity})
is given even in the setting of so-called sparse regular pairs which
is a generalization of the regularity concept which we do not need
in the present paper.
\begin{lem}
\label{lem:randomsubpair}Suppose that $\beta,\epsilon_{1}>0$ are
given. Then there exist numbers $\epsilon_{0}>0$ and $C>0$ such
that for every $\epsilon_{0}$-regular pair $(A,B)$ of an arbitrary
density $d=d(A,B)$ and for any $a,b>\nicefrac{C}{d}$, the number
of pairs $(A',B')$, $A'\subset A$, $B'\subset B$, $|A'|=a$, $|B'|=b$
which induce $\epsilon_{1}$-regular pairs of density at least $d-\epsilon_{1}$
is at least 
\[
\exp(-\beta^{a}-\beta^{b})\binom{|A|}{a}\binom{|B|}{b}\;.
\]
\end{lem}

Last, we need a weak version of the Blow-up Lemma. In our version,
we do not require perfect tiling.
\begin{lem}
\label{lem:weakblowup}Suppose that we are given a graph $F$ on a
vertex set $[k]$ and numbers $\gamma,d>0$. Then there exist numbers
$\epsilon>0$ and $m_{0}\in\mathbb{N}$ such that the following holds.
If $B$ is a graph whose vertices are partitioned into sets $V_{1},\ldots,V_{k}$
with $|V_{1}|=\ldots=|V_{k}|>m_{0}$ and which has the property that
for each $ij\in E(F)$ the pair $(V_{i},V_{j})$ is $\epsilon$-regular
with density at least $d$, then $\TIL(F,B)\ge(1-\gamma)|V_{1}|$.
\end{lem}

\subsection{Subgraphons converging to a subgraphon}

In this section, we provide a proof of the following simple lemma
(which seems to be new).
\begin{lem}
\label{lem:subgraphonscoverge}Suppose that $\left(W_{n}\right)_{n}$
is a sequence of graphons on $\Omega$ converging to a graphon $W:\Omega^{2}\rightarrow[0,1]$
in the cut-norm. Let $U\le W$ be an arbitrary graphon. Then there
exists a sequence $\left(U_{n}\right)_{n}$ with $U_{n}\le W_{n}$
which converges to $U$ in the cut-norm.
\end{lem}

\begin{proof}
Let $\epsilon>0$ be arbitrary. There exist $k$ and a partition of
$\Omega$ into sets of measure $\nicefrac{1}{k}$ each, $\Omega=\Omega_{1}\sqcup\ldots\sqcup\Omega_{k}$,
such that there exist step-functions with steps $\Omega_{i}\times\Omega_{j}$
which approximate $U$ and $W$ up to an error at most $\epsilon^{2}$
in the $\mathcal{L}^{1}(\Omega^{2})$-norm (each one separately).
Such partition and approximation exist since squares generate the
sigma-algebra on $\Omega^{2}$. It is also well known that the said
step-functions can be taken to be constant $U^{ij}:=k^{2}\int_{(x,y)\in\Omega_{i}\times\Omega_{j}}U(x,y)\diffr\nu^{2}$
and $W^{ij}:=k^{2}\int_{(x,y)\in\Omega_{i}\times\Omega_{j}}W(x,y)\diffr\nu^{2}$
on each square $\Omega_{i}\times\Omega_{j}$. A routine calculation
gives that for all but at most $2\epsilon k^{2}$ pairs $(i,j)$ we
have
\begin{equation}
\int_{(x,y)\in\Omega_{i}\times\Omega_{j}}|U(x,y)-U^{ij}|\diffr\nu^{2}\le\nicefrac{\epsilon}{k^{2}}\quad\mbox{and}\quad\int_{(x,y)\in\Omega_{i}\times\Omega_{j}}|W(x,y)-W^{ij}|\diffr\nu^{2}\le\nicefrac{\epsilon}{k^{2}}\;.\label{eq:approxUandW}
\end{equation}
Let $\mathcal{P}\subset[k]^{2}$ be the set of the pairs that fail~(\ref{eq:approxUandW}).

Since $U^{ij}\le W^{ij}$ for each $i$ and $j$, when we define $U_{n}$
by
\[
U_{n}(x,y)=W_{n}(x,y)\cdot\frac{U^{ij}}{W^{ij}}\mbox{ for \ensuremath{(x,y)\in\Omega_{i}\times\Omega_{j},}}
\]
we indeed get a graphon with $U_{n}\le W_{n}$. Let us now bound $\|U-U_{n}\|_{\square}$.
Let $A,B\subset\Omega$ be arbitrary. We write $R^{ij}=(A\times B)\cap(\Omega_{i}\times\Omega_{j})$.
Then by the triangle inequality we have
\begin{align}
\left|\int_{(x,y)\in A\times B}U(x,y)\diffr\nu^{2}-\int_{(x,y)\in A\times B}U_{n}\diffr\nu^{2}\right| & \le\sum_{i,j}\left|\int_{(x,y)\in R^{ij}}U(x,y)\diffr\nu^{2}-\int_{(x,y)\in R^{ij}}U_{n}(x,y)\diffr\nu^{2}\right|\nonumber \\
 & =\sum_{i,j}\left|\int_{(x,y)\in R^{ij}}U(x,y)\diffr\nu^{2}-\int_{(x,y)\in R^{ij}}W_{n}(x,y)\cdot\frac{U^{ij}}{W^{ij}}\diffr\nu^{2}\right|\;.\label{eq:importantcalculation}
\end{align}
The total contribution of the pairs $(i,j)\in\mathcal{P}$ to the
right-hand side of~(\ref{eq:importantcalculation}) is at most $2\epsilon$.
Suppose now that $(i,j)\not\in\mathcal{P}$, and let us find an upper
bound for the corresponding term in~(\ref{eq:importantcalculation}).
\begin{eqnarray*}
 &  & \left|\int_{(x,y)\in R^{ij}}U(x,y)\diffr\nu^{2}-\int_{(x,y)\in R^{ij}}W_{n}(x,y)\cdot\frac{U^{ij}}{W^{ij}}\diffr\nu^{2}\right|\\
 & \overset{\eqref{eq:approxUandW}}{\le} & \left|\int_{(x,y)\in R^{ij}}U^{ij}(x,y)\diffr\nu^{2}-\int_{(x,y)\in R^{ij}}W_{n}(x,y)\cdot\frac{U^{ij}}{W^{ij}}\diffr\nu^{2}\right|+\nicefrac{\epsilon}{k^{2}}\\
 & \le & U^{ij}\left|\int_{(x,y)\in R^{ij}}1\diffr\nu^{2}-\int_{(x,y)\in R^{ij}}\frac{W_{n}(x,y)}{W^{ij}}\diffr\nu^{2}\right|+\nicefrac{\epsilon}{k^{2}}\\
 & \le & U^{ij}\left|\int_{(x,y)\in R^{ij}}1-\frac{1}{W^{ij}}\left(\|W_{n}-W\|_{\square}+\int_{(x,y)\in R^{ij}}W(x,y)\diffr\nu^{2}\right)\right|+\nicefrac{\epsilon}{k^{2}}\\
 & = & U^{ij}\left|\frac{1}{W^{ij}}\left(\|W_{n}-W\|_{\square}+\int_{(x,y)\in R^{ij}}\left(W(x,y)-W^{ij}\right)\diffr\nu^{2}\right)\right|+\nicefrac{\epsilon}{k^{2}}\\
 & \overset{\eqref{eq:approxUandW}}{\le} & U^{ij}\left(\frac{1}{W^{ij}}\|W_{n}-W\|_{\square}+\nicefrac{\epsilon}{k^{2}}\right)+\nicefrac{\epsilon}{k^{2}}\\
 & \le & \|W_{n}-W\|_{\square}+\nicefrac{2\epsilon}{k^{2}}\;.
\end{eqnarray*}
Thus, the total contribution of the terms $(i,j)\notin\mathcal{P}$
is at most $k^{2}\|W_{n}-W\|_{\square}+2\epsilon$ which is less than
$3\epsilon$ for large enough $n$. Consequently, $\left|\int_{(x,y)\in A\times B}U\diffr\nu^{2}-\int_{(x,y)\in A\times B}U_{n}\diffr\nu^{2}\right|\le5\epsilon$.

Thus, the proof proceeds by letting $\epsilon$ go to zero and diagonalizing.
\end{proof}

\section{The theory of tilings in graphons \label{sec:StatementOfTheResults}}

\subsection{A naive approach}

We want to gain some understanding what an $F$-tiling in a graphon
should be. Let us take our first motivation from the world of finite
graphs. Suppose that $F$ and $G$ are finite graphs, $F$ is on the
vertex set $[k]$. Then each $F$-tiling can be viewed as a set $\mathbf{X}\subset V(G)^{k}$
such that 
\begin{enumerate}[label=(\emph{\alph*})]
\item \label{enu:edges}for each $T\in\mathbf{X}$ we have that $T_{i}T_{j}$
forms an edge of $G$ for each $ij\in E(F)$, $i<j$, and
\item for each $v\in V(G)$ there is at most one pair $(T,i)\in\mathbf{X}\times[k]$
such that $v=T_{i}$.
\end{enumerate}
The size of the $F$-tiling $\mathbf{X}$ is obviously $|\mathbf{X}|$,
which can be rewritten using the projection $\pi_{1}:V(G)^{k}\rightarrow V(G)$
on the first coordinate as
\begin{equation}
|\mathbf{X}|=\left|\pi_{1}(\mathbf{X})\right|\;.\label{eq:usefirstprojection}
\end{equation}
 Let $A$ be the adjacency matrix of $G$. Then Condition~\ref{enu:edges}
can be rewritten as 
\begin{enumerate}[label=(\emph{\alph*}')]
\item for each $T\in\mathbf{X}$ we have that $\prod_{ij\in E(F),i<j}A_{T_{i}T_{j}}>0$.
\end{enumerate}
These conditions can be translated directly to graphons. That is,
we might want to say that a set $\mathbf{Y}\subset\Omega^{k}$ is
an $F$-tiling in a graphon $W:\Omega^{2}\rightarrow[0,1]$ if
\begin{itemize}
\item for each $1\le i<j\le k$ the map $\pi_{j}\circ\pi_{i}^{-1}$ is a
$\nu$-measure preserving bijection between $\pi_{i}(\mathbf{Y})$
and $\pi_{j}(\mathbf{Y})$,
\item for each $T\in\mathbf{Y}$ we have that $\prod_{ij\in E(F),i<j}W(T_{i},T_{j})>0$,
\item and for each $x\in\Omega$ there is at most one pair $(T,i)\in\mathbf{Y}\times[k]$
such that $x=T_{i}$. 
\end{itemize}
Using~(\ref{eq:usefirstprojection}), we would then say that $\nu(\pi_{1}(\mathbf{Y}$))
is the size of the $F$-tiling $\mathbf{Y}$. 

There is however a substantial problem with this approach in that
whether the condition $\prod_{ij\in E(F),i<j}W(T_{i},T_{j})>0$ holds
depends on values of $W$ of $\nu^{2}$-measure zero. The theory of
graphons cannot in principle achieve this level on sensitivity. Thus,
a different approach is needed.

\subsection{Introducing tilings in graphons\label{subsec:IntroducingTilings}}

As we saw in the previous section the main problem with the straightforward
graphon counterpart to $F$-tilings of finite graphs was that the
limiting object was too centralized; so centralized that all the information
was needed to be carried on a set of measure zero. Observe that in
a setting of a finite graph $G$, fractional $F$-tilings can be more
``spread out'' than integer $F$-tilings in the sense that the weight
of the latter is always supported on at most $\nicefrac{v(G)}{v(F)}$
many copies of $F$ while it can be supported on up to $\binom{v(G)}{v(F)}$
many copies in the former. Thus, the notion of fractional $F$-tilings
seems better admissible to a limit counterpart. Indeed, as we shall
see in Section~\ref{subsec:TilingsVSFractionalTilings}, in the world
of graphons there is no difference between $F$-tilings and fractional
$F$-tilings. Thus, we want to approach the notion of $F$-tilings
in graphons by looking at fractional $F$-tilings in finite graphs.

Suppose that $F$ and $G$ are finite graphs, $F$ is on the vertex
set $[k]$, the order of $G$ is $n$, and let $\mathfrak{t}_{G}:\mathcal{F}_{F}(G)\rightarrow[0,1]$
be a fractional $F$-tiling in $G$. Let $W_{G}:\Omega^{2}\rightarrow[0,1]$
be a graphon representation of $G$. Let $\left(\Omega_{u}\right){}_{u\in V(G)}$
be the partition of $\Omega$ corresponding to this representation.
We can then associate a measure $\mu$ on $\Omega^{k}$ corresponding
to $\mathfrak{t}_{G}$ by the defining formula
\begin{equation}
\mu(A)=\sum_{u_{1}u_{2}\cdots u_{k}\in\mathcal{F}_{F}(G)}\frac{\mathfrak{t}_{G}\left(u_{1}u_{2}\cdots u_{k}\right)}{n}\cdot\frac{\nu^{k}\left(\left(\Omega_{u_{1}}\times\Omega_{u_{2}}\times\cdots\times\Omega_{u_{k}}\right)\cap A\right)}{\nu^{k}\left(\Omega_{u_{1}}\times\Omega_{u_{2}}\times\cdots\times\Omega_{u_{k}}\right)}\;,\label{eq:tilingmeasure}
\end{equation}
for each measurable $A\subset\Omega^{k}$. In words, we introduce
a linear renormalization on $\mathfrak{t}_{G}$. Then the measure
$\mu$ spreads the total of $\frac{\mathfrak{t}_{G}\left(u_{1}u_{2}\cdots u_{k}\right)}{n}$
uniformly over each rectangle $\Omega_{u_{1}}\times\Omega_{u_{2}}\times\cdots\times\Omega_{u_{k}}$.
An example is given in Figure~\ref{fig:TilingMeasure}. 
\begin{figure}
\includegraphics[width=0.8\textwidth]{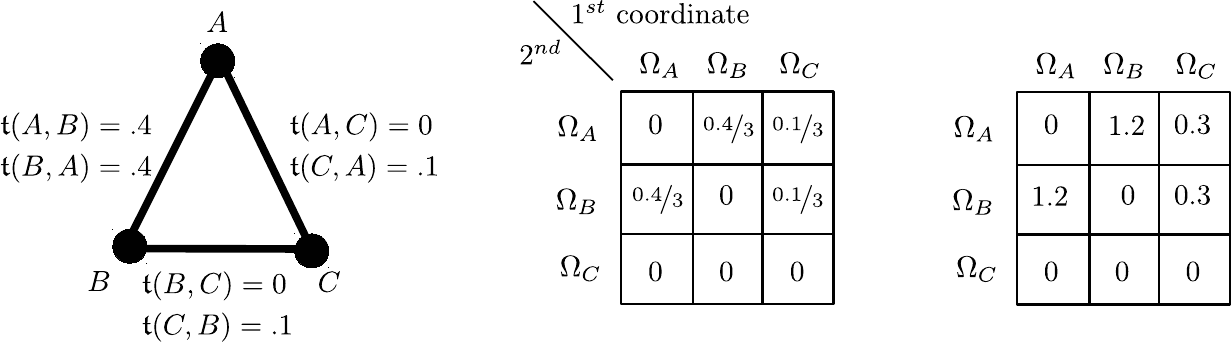}\caption{The left-hand picture shows a sample graph $G$ and a fractional $K_{2}$-tiling
in it. The middle picture shows the total masses of the nine rectangles
with respect to the measure $\mu$ defined by~(\ref{eq:tilingmeasure}).
The right-hand picture then shows the point-wise values of the Radon\textendash Nikodym
derivative $\frac{\mathrm{d}\mu}{\mathrm{d}\nu^{2}}$.}
\label{fig:TilingMeasure}
\end{figure}
The following properties of $\mu$ are obvious:
\begin{enumerate}
\item $\mu$ is absolutely continuous with respect to $\nu^{k}$,
\item $\mu$ is supported on a subset of the set $\text{\ensuremath{\mathcal{F}}}_{F}(W_{G})$,
\item for each $X\subset\Omega$ we have $\sum_{\ell=1}^{k}\mu\left(\prod_{i=1}^{\ell-1}\Omega\times X\times\prod_{i=\ell+1}^{k}\Omega\right)\le\nu(X)$. 
\end{enumerate}
Property~(1) allows us to take the Radon\textendash Nikodym derivative
$\frac{\mathrm{d}\mu}{\mathrm{d}\nu^{k}}$ of $\mu$. We define functions
that arise in this way as $F$-tilings.\footnote{As said, later we shall see that there is no distinction between integral
and fractional tilings for graphons. Thus, even though the motivation
for the current notion comes from \emph{fractional} tilings in finite
graphs, we call the resulting graphon concept simply \emph{tilings}.}
\begin{defn}
\label{def:graphontiling}Suppose that $W:\Omega^{2}\rightarrow[0,1]$
is a graphon, and that $F$ is a graph on the vertex set $[k]$. An
$\mathcal{L}^{1}(\Omega^{k})$-function $\mathfrak{t}:\Omega^{k}\rightarrow[0,+\infty)$
is called an \emph{$F$-tiling} in $W$ if 
\begin{equation}
\text{\ensuremath{\SUPPORT}\ensuremath{\mathfrak{t}\subset}\ensuremath{\mathcal{F}}}_{F}(W)\;,\label{eq:tilingsupport}
\end{equation}
 and we have for each $x\in\Omega$ that 
\begin{equation}
\sum_{\ell=1}^{k}\int_{(x_{1},\ldots x_{\ell-1},x_{\ell+1},\ldots,x_{k})\in\Omega^{k-1}}\mathfrak{t}(x_{1},\ldots,x_{\ell-1},x,x_{\ell+1},\ldots,x_{k})\diffr\nu^{k-1}\le1\;.\label{eq:IntegralAtMostOne}
\end{equation}
The \emph{size} of an $F$-tiling $\mathfrak{t}$ is $\|\mathfrak{t}\|=\int_{(x_{1},\ldots,x_{k})\in\Omega^{k}}\mathfrak{t}(x_{1},\ldots,x_{k})\diffr\nu^{k}$.
The \emph{$F$-tiling number} of $W$, denoted by $\TIL(F,W)$, is
the supremum of sizes over all $F$-tilings in $W$.
\end{defn}

A couple of remarks is in place. First, if $F$ and $F'$ are copies
of the same graph on the vertex set $[k]$ with the vertex-labels
permuted we can get an $F'$-tiling from an $F$-tiling by permuting
the corresponding coordinates. Second, observe that the definition
does not depend on the values of $W$, only on $\SUPPORT W$. Third,
we have $\TIL(F,W)\le\frac{1}{k}$. This is a counterpart to the finite
statement $\TIL(F,G)\le\frac{v(G)}{k}$. Fourth, the supremum in the
definition of $\TIL(F,W)$ need not be attained. An example is given
in Proposition~\ref{prop:examplenotattained}. Fifth, when $W_{G}$
is a representation of a finite graph $G$, we have $\FTIL(F,G)=v(G)\TIL(F,W_{G})$
(we emphasize that the relevant graph parameter is $\FTIL(F,G),$
not $\TIL(F,G)$). While this would easily follow from tools we develop
later (specifically, from Proposition~\ref{prop:FCOVgraphFCOV} and
Theorem~\ref{thm:LPdualityGraphons}), here we give a self-contained
proof.
\begin{prop}
\label{prop:FTILgraphTILgraphon}Suppose that $F$ and $G$ are finite
graphs and that $W_{G}$ is a graphon representation of $G$. Let
$n$ be the number of vertices of $G$. Then we have $\FTIL(F,G)=n\cdot\TIL(F,W_{G})$.
\end{prop}

\begin{proof}
Let us assume that the vertex set of $G$ is $V(G)=[n]$. Let us consider
the partition $\Omega=\Omega_{1}\sqcup\ldots\sqcup\Omega_{n}$ of
the space that hosts $W_{G}$ into sets representing the individual
vertices of~$G$.

Suppose that $\mathfrak{t}_{G}:\mathcal{F}_{F}(G)\rightarrow[0,1]$
is an arbitrary fractional $F$-tiling in $G$. Define $\mathfrak{t}:\Omega^{k}\rightarrow[0,+\infty)$
to be constant $\mathfrak{t}_{G}(i_{1},i_{2},\cdots,i_{k})/n$ on
each set $\Omega_{i_{1}}\times\Omega_{i_{2}}\times\cdots\cdots\times\Omega_{i_{k}}$.
It is straightforward to check that $\mathfrak{t}$ is an $F$-tiling
in $W$ of size $\|\mathfrak{t}\|/n$. We conclude that $\FTIL(F,G)\le n\cdot\TIL(F,W_{G})$.

Suppose that $\mathfrak{t}:\Omega^{k}\rightarrow[0,+\infty)$ is an
arbitrary $F$-tiling in $W$. Define $\mathfrak{t}_{G}:\mathcal{F}_{F}(G)\rightarrow[0,1]$
to be constant $n\cdot\int_{(x_{1},\cdots,x_{k})\in\Omega_{i_{1}}\times\Omega_{i_{2}}\times\cdots\times\Omega_{i_{k}}}\mathfrak{t}(x_{1},\cdots,x_{k})\diffr\nu^{k}$
for each $(i_{1},i_{2},\ldots,i_{k})\in\mathcal{F}_{F}(G)$. It is
straightforward to check that $\mathfrak{t}_{G}$ is an $F$-tiling
in $W$ of size $n\|\mathfrak{t}\|$. We conclude that $\FTIL(F,G)\ge n\cdot\TIL(F,W_{G})$.
\end{proof}
\begin{prop}
\label{prop:examplenotattained}Consider the graphon $W:[0,1]^{2}\rightarrow[0,1]$
defined as 
\begin{equation}
W(x,y)=\begin{cases}
0 & \mbox{if }x+y>\nicefrac{1}{2}\\
1 & \mbox{if }x+y\le\nicefrac{1}{2}
\end{cases}\;.\label{eq:HalfGraphon}
\end{equation}
Then $\TIL(K_{2},W)=\nicefrac{1}{2}$, but there exists no $K_{2}$-tiling
of size $\nicefrac{1}{2}$.
\end{prop}

\begin{proof}
For an arbitrary $C\ge1$ we can take a function $\mathfrak{t}_{C}:[0,1]^{2}\rightarrow[0,+\infty)$
defined by 
\[
\mathfrak{t}_{C}(x,y)=\begin{cases}
0 & \mbox{if }x+y>\nicefrac{1}{2}\mbox{ or }x+y<\nicefrac{1}{2}-\nicefrac{1}{C}\\
C & \mbox{if }x+y\in\left[\nicefrac{1}{2}-\nicefrac{1}{C},\nicefrac{1}{2}\right]
\end{cases}\;.
\]
It is easy to see that $\mathfrak{t}_{C}$ is a $K_{2}$-tiling and
that $\mathfrak{\|t}_{C}\|\ge\nicefrac{1}{2}-\nicefrac{1}{C}$. Thus
$\TIL(F,W)\ge\nicefrac{1}{2}$. Yet, there is no $K_{2}$-tiling of
size $\nicefrac{1}{2}$. To see this, assume for a contradiction that
$\mathfrak{t}$ is such a $K_{2}$-tiling. By replacing $\mathfrak{t}(x,y)$
by $\frac{1}{2}\left(\mathfrak{t}(x,y)+\mathfrak{t}(y,x)\right)$,
we can assume that $\mathfrak{t}(x,y)$ is symmetric. To get the contradiction,
it is enough to show that $\mathfrak{t}$ is constant zero almost
everywhere on each rectangle $[0,\alpha]\times[0,1-\alpha]$ (see
Figure~\ref{fig:EmptyRectangle}). Firstly, observe that because
$t$ is symmetric, (\ref{eq:IntegralAtMostOne}) tells us that $\int_{y\in[0,1]}\mathfrak{t}(x_{0},y)\diffr y\le\nicefrac{1}{2}$
for every $x_{0}\in[0,1]$. Since $\nicefrac{1}{2}=\|\mathfrak{t}\|=\int_{x_{0}\in[0,1]}\left(\int_{y\in[0,1]}\mathfrak{t}(x_{0},y)\diffr y\right)\diffr x_{0}$,
we conclude that $\int_{y\in[0,1]}\mathfrak{t}(x_{0},y)\diffr y=\nicefrac{1}{2}$
for almost every $x_{0}\in[0,1]$. Applying this for $x_{0}\ge1-\alpha$,
we get that the integral $\int_{x\in[1-\alpha,1]}\left(\int_{y\in\left[0,x\right]}\mathfrak{t}(x,y)\diffr y\right)\diffr x$
(over the green triangle in Figure~\ref{fig:EmptyRectangle}) is
exactly $\nicefrac{\alpha}{2}$. The value $\int_{x\in[0,1]}\left(\int_{y\in[0,\alpha]}\mathfrak{t}(x,y)\diffr y\right)\diffr x$
is at most $\nicefrac{\alpha}{2}$ by~(\ref{eq:IntegralAtMostOne}).
We conclude that $\mathfrak{t}$ is zero almost everywhere on $[0,\alpha]\times[0,1-\alpha]$.
\begin{figure}
\includegraphics[scale=0.8]{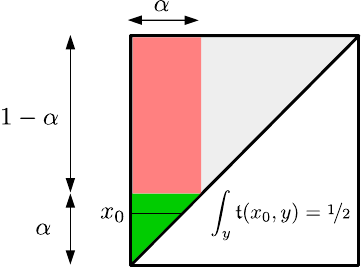}\caption{For the graphon defined by (\ref{eq:HalfGraphon}), any symmetric
$K_{2}$-tiling $\mathfrak{t}$ of weight $\|\mathfrak{t}\|=\nicefrac{1}{2}$
satisfies that $\mathfrak{t}$ is zero almost everywhere on the red
rectangle $[0,\alpha]\times[0,1-\alpha]$. }
\label{fig:EmptyRectangle}
\end{figure}
\end{proof}

\subsection{Graphon tilings versus fractional graphon tilings\label{subsec:TilingsVSFractionalTilings}}

Let us now explain that there should be no difference between $F$-tilings
and fractional $F$-tilings for graphons. In the world of finite graphs,
the former is a proper subset of the latter. So, let us show that
in the graphon world we have the opposite inclusion as well. That
is, we want to show that a fractional $F$-tiling $\mathfrak{t}$
of certain size in $W$ yields an $F$-tiling of approximately the
same size (after rescaling) in finite graphs $G_{n}$ of a sequence
that converges to $W$. To this end, fix a sequence $\left(R_{n}\right)_{n}$
of finer and finer Szemerédi regularizations of the graphs $G_{n}$.
That is, $R_{n}$ is a cluster graph whose edges carry weights in
the interval $[0,1]$ and which comes from regularizing $G_{n}$ with
an error $\epsilon_{n}$, with $\lim_{n}\epsilon_{n}=0$. Since the
sequence $\left(R_{n}\right)_{n}$ converges to $W$ in the cut-distance
we can take $\mathfrak{t}$ and turn into a fractional $F$-tiling
$\mathfrak{t}_{n}$ in $R_{n}$ of size $\approx\|\mathfrak{t}\|v(R_{n})$
(for $n$ large). Of course, the fact that we can transfer a fractional
$F$-tiling on a graphon to a fractional $F$-tiling on a graph in
a sequence converging to this graphon needs a formal statement, which
we give in Theorem~\ref{thm:lowersemicontgraphs}. In the last step,
we recall that Lemma~\ref{lem:weakblowup} provides a standard tool
for finding an $F$-tiling of size~$\approx\|\mathfrak{t}\|n$ in~$G_{n}$
based on an fractional $F$-tiling of size $\approx\|\mathfrak{t}\|v(R_{n})$
in the cluster graph $R_{n}$. A description by picture how this step
is done is given in Figure~\ref{fig:blowup}.
\begin{figure}
\includegraphics[width=0.9\textwidth]{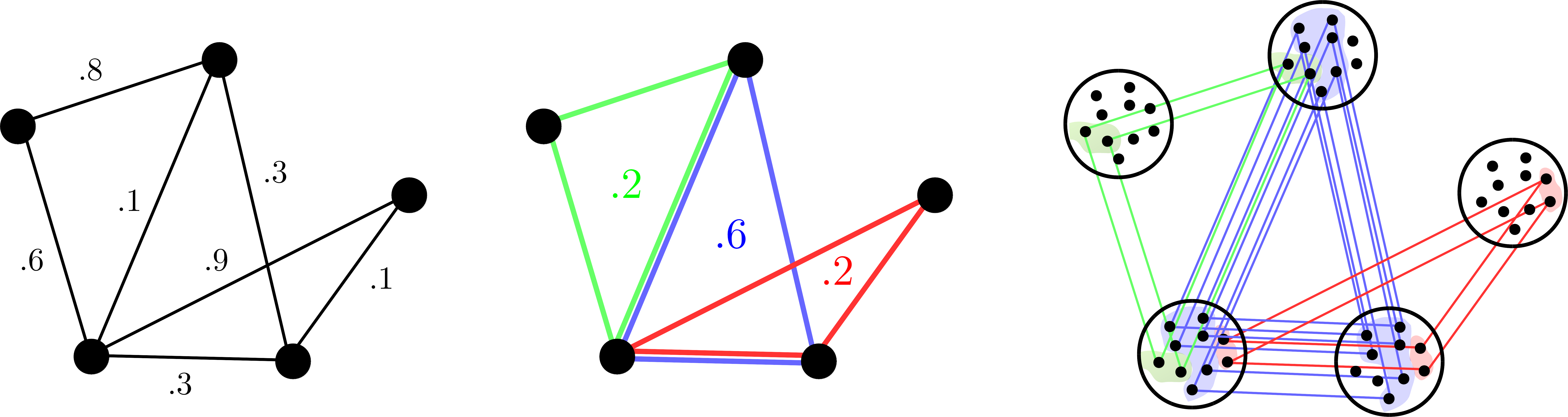}\caption{The left-hand picture shows the graph $R_{n}$ with the edge-weights.
The middle picture shows a fractional triangle-tiling $\mathfrak{t}$
in $R_{n}$ (the different copies of a triangle are depicted with
different colors). On the right-hand picture, we split the clusters
of $G_{n}$ according to $\mathfrak{t}$, thus obtaining subpairs
of the original regular pairs. Using Lemma~\ref{lem:weakblowup},
we get an $F$-tiling in $G_{n}$ of size proportional to the size
of the fractional $F$-tiling in $R_{n}$. (This example is oversimplified:
If the triangle-tiling is ``spread out'', meaning that $\mathfrak{t}(H)$
is of order $\epsilon_{n}$ or less for many copies $H$ of the triangle,
then the regularity of the pairs is not inherited to the subpairs
and Lemma~\ref{lem:weakblowup} cannot be applied. As we shall see,
this issue can be resolved.)}
\label{fig:blowup}
\end{figure}

Note that the above also explains our remark in Section~\ref{subsec:IntroducingTilings}
that the notion of tilings depends only on the support of the graphon,
and not on its values themselves. Indeed, Lemma~\ref{lem:weakblowup}
was applied to subpairs of regular pairs. And Lemma~\ref{lem:weakblowup}
works for regular pairs of an arbitrary positive density.\footnote{Lower density will require a finer error parameter of regularity.
But in the limit case, we have ``infinitely fine regular pairs''.}

\subsection{Tilings and convergence}

Suppose $F$ is a finite graph and that we have a sequence $(G_{n})$
of graphs, $v(G_{n})=n$, that converges to a graphon $W:\Omega^{2}\rightarrow[0,1]$
in the cut-distance. Observe that the numbers $\TIL(F,G_{n})$ grow
(up to) linearly in $n$. Hence, we could hope that $\frac{\TIL(F,G_{n})}{n}\to\TIL(F,W)$
holds. Unfortunately, this is not always true. Indeed, consider $F=K_{2}$,
and the graphs~$G_{n}$ are defined as a perfect matching~on $n$
vertices for~$n$ even, and as an edgeless $n$-vertex graph for
$n$ odd. Then $\left(G_{n}\right)_{n}$ converges to the zero graphon
but $\frac{\TIL(F,G_{n})}{n}$ oscillates between $0$ and $\frac{1}{2}$.
So, while we see that the quantity $\TIL(F,\cdot)$ cannot be continuous
on the space of graphons, here we establish lower semicontinuity for
graphs converging to graphon.
\begin{thm}
\label{thm:lowersemicontgraphs}Suppose that $F$ is a finite graph
and let $(G_{n})$ be a sequence of graphs of growing orders converging
to a graphon $W:\Omega^{2}\rightarrow[0,1]$ in the cut-distance.
Then we have that $\liminf_{n}\frac{\TIL(F,G_{n})}{v(G_{n})}\ge\TIL(F,W)$.
\end{thm}

The proof of Theorem~\ref{thm:lowersemicontgraphs} is given in Section~\ref{subsec:ProofTheoremLowesemicontinuousGraphs}.
\begin{rem}
\label{rem:lowersemicontinuitybetter}Lower semicontinuity is the
more applicable half of continuity for the purposes of extremal graph
theory. Indeed, in extremal graph theory one typically wants to bound
the $F$-tiling number of graphs from a graph class of interest from
below (as opposed to bounding from above). Thanks to Theorem~\ref{thm:lowersemicontgraphs}
this can be achieved (in the asymptotic sense) by finding a lower
bound $\TIL(F,W)$ for each limiting graphon $W$. This proof scheme
is used in~\cite{HlHuPi:Komlos}.
\end{rem}

\begin{figure}
\includegraphics{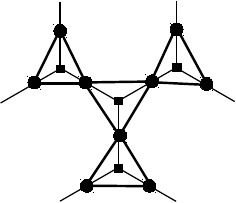}\caption{The transformation described in Remark~\ref{rem:compareBenjaminiSchramm}.
A part of the random (bipartite) graph $H_{n}$ shown with thin lines
and box vertices. The corresponding line graph~$G_{n}$ shown with
thick lines and circle vertices.}
\label{fig:transformation}
\end{figure}

\begin{rem}
\label{rem:compareBenjaminiSchramm}Let us compare the situation with
the sparse case. That is, suppose that we have a Benjamini\textendash Schramm
convergent sequence $(G_{n})_{n}$ of bounded degree graphs, and we
are concerned with the sequence $\left(\nicefrac{\TIL(F,G_{n})}{v(G_{n})}\right)_{n}$
for some fixed graph $F$. A theorem of Nguyen and Onak~\cite{NO:ConstantTime}
(reproved later by Elek and Lippner~\cite{EL:BorelOracles} and by
Bordanave, Lelarge, and Salez~\cite{BLS:MatchingOnInfiniteGraphs})
tells us that indeed the normalized matching ratios (i.e., the most
important case $F=K_{2}$) converge. In contrast, Endre Csóka has
communicated to us that the sequence $\left(\nicefrac{\TIL(K_{3},G_{n})}{v(G_{n})}\right)_{n}$
needs not to converge. The construction he uses to this end is as
follows. For $n$ even, consider a random cubic graph $H_{n}$ of
order $2n$, for $n$ odd consider a random bipartite cubic graph
$H_{n}$ of order $2n$. It is well-known that the resulting sequence
is Benjamini\textendash Schramm convergent almost surely, and that
the limit is the rooted infinite 3-regular tree. Let $G_{n}$ be the
line graphs of $H_{n}$. This is shown in Figure~\ref{fig:transformation}.
There is a one-to-one correspondence between independent sets in $H_{n}$'s
and triangle tilings in $G_{n}$'s. Thus, when $H_{n}$ is bipartite
($n$ even), we have $\TIL(K_{3},G_{n})=v(G_{n})/3$. On the other
hand, a result of Bollobás about independence number of random cubic
graphs~\cite{Bol:IndependentSets} translates as $\TIL(K_{3},G_{n})\le0.92\cdot v(G_{n})/3$
asymptotically almost surely for $n$ odd. 

A similar construction shows that the parameter $\nicefrac{\TIL(F,G)}{v(G)}$
is discontinuous in the Benjamini\textendash Schramm topology for
each 2-connected graph $F$. In that construction, random cubic graphs
and random bipartite cubic graphs are replaced by random $v(F)$-regular
graphs and random bipartite $v(F)$-regular graphs, respectively.
It would be interesting to fully characterize the graphs $F$ for
which the quantity $\nicefrac{\TIL(F,\cdot)}{v(\cdot)}$ is continuous.
The case when $F$ is a path or a star seems particularly interesting.
\end{rem}

In Theorem~\ref{thm:lowersemicont} below we state a version of Theorem~\ref{thm:lowersemicontgraphs}
for a sequence $(W_{n})$ of graphons converging to a graphon $W$.
Then we have $\liminf_{n}\TIL(F,W_{n})\ge\TIL(F,W)$. Note that this
is not a strengthening of Theorem~\ref{thm:lowersemicontgraphs}
since in general we do not have (even approximately) that $\frac{\TIL(F,G)}{v(G)}=\TIL(F,U_{G})$
for a graphon representation $U_{G}$ of the graph $G$.
\begin{thm}
\label{thm:lowersemicont}Suppose that $(W_{n})$ is a sequence of
graphons $W_{n}:\Omega^{2}\rightarrow[0,1]$ converging to a graphon
$W:\Omega^{2}\rightarrow[0,1]$ in the cut-distance and let $F$ be
an arbitrary graph. Then we have that $\liminf_{n}\TIL(F,W_{n})\ge\TIL(F,W)$.
\end{thm}

The proof of Theorem~\ref{thm:lowersemicont} is given in Section~\ref{subsec:ProofTheoremLowesemicontinuous}.

\subsection{Robust tiling number}

While \textemdash{} as we have explained \textemdash{} the function
$\TIL(F,\cdot)$ is not upper-semicontinuous, it is ``upper-semicontinuous
in the cut-distance after small $\mathcal{L}^{1}$-perturbations''.
This is stated below.
\begin{prop}
\label{prop:uppercont}Suppose that $F$ is an arbitrary graph and
that $W:\Omega^{2}\rightarrow[0,1]$ is a graphon. Then for an arbitrary
$\eta>0$ there exists a number $\delta>0$ such that each graphon
$U$ with $\left\Vert W-U\right\Vert _{\square}<\delta$ can be decreased
in the $\mathcal{L}^{1}(\Omega^{2})$-distance by at most $\eta$
(in a suitable way) so that we obtain a graphon $U^{*}$ for which
$\TIL(F,U^{*})\le\TIL(F,W)+\eta$.
\end{prop}

The proof of Proposition~\ref{prop:uppercont} is given in Section~\ref{subsec:ProofPropUpperCont}.
Let us state a version of Proposition~\ref{prop:uppercont} which
deals with the situation when the the graphon $W$ is approximated
by a finite graph and not a graphon.
\begin{prop}
\label{prop:uppercontGraph}Suppose that $F$ is an arbitrary graph
and that $W:\Omega^{2}\rightarrow[0,1]$ is a graphon. Then for an
arbitrary $\eta>0$ there exists a number $\delta>0$ such that in
each finite graph $G$ with $\DIST_{\square}(W,G)<\delta$ we can
erase at most $\eta v(G)^{2}$ edges (in a suitable way) so that for
the resulting graph $G^{*}$ we have $\TIL(F,G^{*})\le n\cdot\left(\TIL(F,W)+\eta\right)$.
\end{prop}

Proposition~\ref{prop:uppercontGraph} is not stronger than Proposition~\ref{prop:uppercont}
since it deals with graphs only. On the other hand, Proposition~\ref{prop:uppercontGraph}
does not follow directly from Proposition~\ref{prop:uppercont} since
it contains the extra assertion that each edge is either entirely
deleted, or kept entirely untouched. We do not include a proof of
Proposition~\ref{prop:uppercontGraph} as it is analogous to that
of Proposition~\ref{prop:uppercont}.

Propositions~\ref{prop:uppercont} motivates the following definition.
Given a finite graph $F$, a number $\epsilon>0$ and a graphon $W:\Omega^{2}\rightarrow[0,1]$
we define 
\begin{equation}
\TIL_{\epsilon}(F,W)=\inf_{W^{-}}\TIL\left(F,W^{-}\right)\;,\label{eq:defTILrobust}
\end{equation}
where $W^{-}$ ranges over all graphons with $W^{-}\le W$ and with
$\left\Vert W-W^{-}\right\Vert _{1}\le\epsilon$. Similarly, for an
$n$-vertex graph $G$, we define
\[
\TIL_{\epsilon}(F,G)=\inf_{G^{-}}\TIL\left(F,G^{-}\right)\;,
\]
where $G^{-}$ ranges over all subgraphs of $G$ with at most $\epsilon n^{2}$
edges deleted from $G$. These ``robust versions of the tiling number''
are continuous.
\begin{thm}
\label{thm:robusttilingnumber}Suppose that $F$ is a finite graph
and $\epsilon>0$. Then the quantity $\TIL_{\epsilon}(F,\cdot)$ is
continuous on the space of graphons equipped with the cut-norm. 
\end{thm}

\begin{thm}
\label{thm:robusttilingnumberGRAPHS}Suppose that $F$ is a finite
graph and $\epsilon>0$. Suppose that $\left(G_{n}\right)_{n}$ is
a sequence of graphs of growing orders that converges in the cut-distance
to a graphon $W$. Then the sequence $\left(\nicefrac{\TIL_{\epsilon}(F,G_{n})}{v(G_{n})}\right)_{n}$
converges to $\TIL_{\epsilon}(F,W)$.
\end{thm}

We shall give a proof of Theorem~\ref{thm:robusttilingnumber} in
Section~\ref{subsec:ProofRobustTiling}. We omit a proof of Theorem~\ref{thm:robusttilingnumberGRAPHS}
but it follows easily by combining the proofs of Theorem~\ref{thm:robusttilingnumber}
and Theorem~\ref{thm:lowersemicontgraphs}.

Let us give an interpretation of Theorem~\ref{thm:robusttilingnumberGRAPHS}
in property testing in the so-called dense model. Formally, let $\mathcal{G}$
be the class of all isomorphism classes of finite graphs. We say that
a function (called often a \emph{parameter}) $f:\mathcal{G}\rightarrow\mathbb{R}$
is \emph{testable} if for each $\epsilon>0$ there exists a number
$r=r(\epsilon)$ and a function (called often a \emph{tester}) $g:\mathcal{G}\rightarrow\mathbb{R}$
such that for each $G\in\mathcal{G}$, we have
\[
\mathbf{P}\left[\left|f(G)-g(H)\right|>\epsilon\right]<\epsilon\;,
\]
where $H=G[A]$ is the subgraph of $G$ induced by selecting a set
$A$ of $r$ vertices at random. A very convenient and concise characterization
of testable parameters was provided in~\cite{Borgs2008c} in the
language of graph limits: \emph{A graph parameter $f$ is testable
if and only if it is continuous in the cut-distance}.\footnote{See also~\cite{Lovasz} for more advanced connections between property
testing and graph limits.} Also, in the positive case, we can take the tester to be $g=f$.

Thus, Theorem~\ref{thm:robusttilingnumberGRAPHS} tells us that $\TIL_{\epsilon}(F,\cdot)$
is testable. For example, suppose that we have a large computer network
$G$ and we want to estimate the size of the largest possible matching
in it. But a small number of links between the computers may become
broken, and we do not know in advance which links these will be, so
we actually want to estimate $\TIL_{\epsilon}(K_{2},G)$. Then a good
estimate to this quantity can be provided just by computing $\TIL_{\epsilon}(K_{2},H)$
for a large, randomly selected induced subgraph $H$ of $G$.

\subsection{Fractional covers and LP duality}

We define fractional covers in graphons in a complete analogy to the
finite notion.
\begin{defn}
\label{def:covergraphon}Suppose that $W:\Omega^{2}\rightarrow[0,1]$
is a graphon, and $F$ is a graph on the vertex set~$[k]$. A measurable
function $\mathfrak{c}:\Omega\rightarrow[0,1]$ is called a \emph{fractional
$F$-cover} in $W$ if
\[
\nu^{k}\left(\text{\ensuremath{\mathcal{F}}}_{F}(W)\cap\left\{ (x_{1},x_{2},\ldots,x_{k})\in\Omega^{k}:\sum_{i=1}^{k}\mathfrak{c}(x_{i})<1\right\} \right)=0\;.
\]

The \emph{size} of $\mathfrak{c}$, denoted by $\|\mathfrak{c}\|$,
is defined by $\|\mathfrak{c}\|=\int_{x\in\Omega}\mathfrak{c}(x)\diffr\nu$.
The \emph{fractional $F$-cover number }$\FCOV(F,W)$ of $W$ is the
infimum of the sizes of fractional $F$-covers in $W$\emph{.}
\end{defn}

As with tilings, let us note that the notion of fractional covers
does not depend on the values of the graphon but only on its support.

The following proposition gives us a simple but useful relation between
the fractional $F$-cover number between a graph and its graphon representation.
\begin{prop}
\label{prop:FCOVgraphFCOV}Suppose that~$F$ and~$G$ are finite
graphs and that~$W_{G}$ is a graphon representation of~$G$. Let~$n$
be the number of vertices of~$G$. Then we have $\FCOV(F,G)=n\cdot\FCOV(F,W_{G})$.
\end{prop}

\begin{proof}
Suppose that $\Omega_{1}\sqcup\Omega_{2}\sqcup\ldots\sqcup\Omega_{n}$
is the partition of the underlying probability space $\Omega$ given
by the graphon representation $W$. Any fractional $F$-cover of $G$
can be represented as a function with steps $\Omega_{1},\Omega_{2},\ldots,\Omega_{n}$.
This step-function is a fractional $F$-cover of $W_{G}$. This shows
that $\FCOV(F,G)\ge n\cdot\FCOV(F,W_{G})$. On the other hand, let
$\mathfrak{c}:\Omega\rightarrow[0,1]$ be an arbitrary fractional
$F$-cover of $W_{G}$. Let us modify $\mathfrak{c}$ on each set
$\Omega_{i}$ by replacing it by its essential infimum on that set.
Since $W_{G}$ is constant on the rectangles $\Omega_{i}\times\Omega_{j}$
we get that the modified step-function $\mathfrak{c}'$ is still a
fractional $F$-cover of $W_{G}$, obviously with $\|\mathfrak{c}'\|\le\|\mathfrak{c}\|$.
But $\mathfrak{c}'$ can be viewed as a fractional $F$-cover of $G$.
This shows that $\FCOV(F,G)\le n\cdot\FCOV(F,W_{G})$.
\end{proof}
The sequence of graphons representing a growing number of isolated
copies of $F$ converges to the constant zero graphon. This shows
that the fractional $F$-cover number is not continuous. We can, however,
establish lower semicontinuity using the following theorem.
\begin{thm}
\label{thm:limitOfCovers}Suppose that $F$ is a graph on the vertex
set $[k]$. Let $\left(W_{n}\right)$ be a sequence of graphons $W_{n}:\Omega^{2}\rightarrow[0,1]$
converging to a graphon $W$ in the cut-norm. Suppose that $\mathfrak{c}_{n}$
is a fractional $F$-cover of $W_{n}$. Then an arbitrary weak$^{*}$accumulation
point $\mathfrak{c}$ of $\left(\mathfrak{c}_{n}\right)$ is a fractional
$F$-cover of $W$.
\end{thm}

We give the proof of Theorem~\ref{thm:limitOfCovers} in Section~\ref{subsec:ProofTheoremLimitOfCovers}.

Recall that if a sequence of functions $\left(f_{n}\right)$ weak$^{*}$
converges to a function $f$, we have $\lim\int f_{n}=\int f$. Also,
recall that by Theorem~\ref{thm:BanachAlauglou} the set of all functions
of $\mathcal{L}^{\infty}$-norm at most 1 is sequentially compact
with respect to the weak$^{*}$ topology. Thus, we get the following.
\begin{cor}
\label{cor:liminfCovers} Suppose that $F$ is a finite graph. Suppose
that $\left(W_{n}\right)$ is a sequence of graphons converging to
$W$. Then $\liminf_{n}\FCOV(F,W_{n})\ge\FCOV(F,W)$.
\end{cor}

Corollary~\ref{cor:liminfCovers} has two important consequences.
Firstly, it tells us that $\FCOV(F,\cdot)$ is lower-semicontinuous
with respect to the cut-distance. Secondly, if we take $W_{1}=W_{2}=\cdots=W$,
and $\mathfrak{c}_{n}$ a sequence of fractional $F$-covers whose
sizes tend to $\FCOV(F,W)$, we get a fractional $F$-tiling which
attains the value $\FCOV(F,W)$. Thus, unlike with $\TIL(F,W)$, 
\begin{equation}
\mbox{the value of \ensuremath{\FCOV(F,W)} is attained by some fractional \ensuremath{F}-cover.}\label{eq:fcovattained}
\end{equation}

We are now ready to state the LP duality theorem for tilings in graphons.
\begin{thm}
\label{thm:LPdualityGraphons}Suppose that $W:\Omega^{2}\rightarrow[0,1]$
is a graphon and $F$ is an arbitrary finite graph. Then we have $\TIL(F,W)=\FCOV(F,W)$.
\end{thm}

We give the proof of Theorem~\ref{thm:LPdualityGraphons} in Section~\ref{subsec:ProofTheoremLPDualityGraphons}.

Theorem~\ref{thm:LPdualityGraphons} is a very convenient tool for
extremal graph theory. More specifically, suppose that we want to
prove that for a fixed graph $F$ and for a family of graphs $\mathcal{G}$
we have that $\TIL(F,G)\ge(\gamma+o(1))v(G)$ for each $G\in\mathcal{G}$.
(Here, $o(1)$ tends to 0 as $v(G)$ goes to infinity.) By combining
Theorem~\ref{thm:lowersemicontgraphs} and Theorem~\ref{thm:LPdualityGraphons}
it suffices to show that no graphon arising as a limit of graphs from
$\mathcal{G}$ has a fractional $F$-cover of size less than $\gamma$.
We shall see one particular application of this scheme in Section~\ref{sec:RandomGraphs}
and another one is used in the accompanying paper~\cite{HlHuPi:Komlos}
on Komlós's Theorem.

\section{Proofs\label{sec:Proofs}}

Figure~\ref{fig:proofStruct} shows the dependencies and the key
steps in the proofs of our main results. 
\begin{figure}
\includegraphics[width=0.8\textwidth]{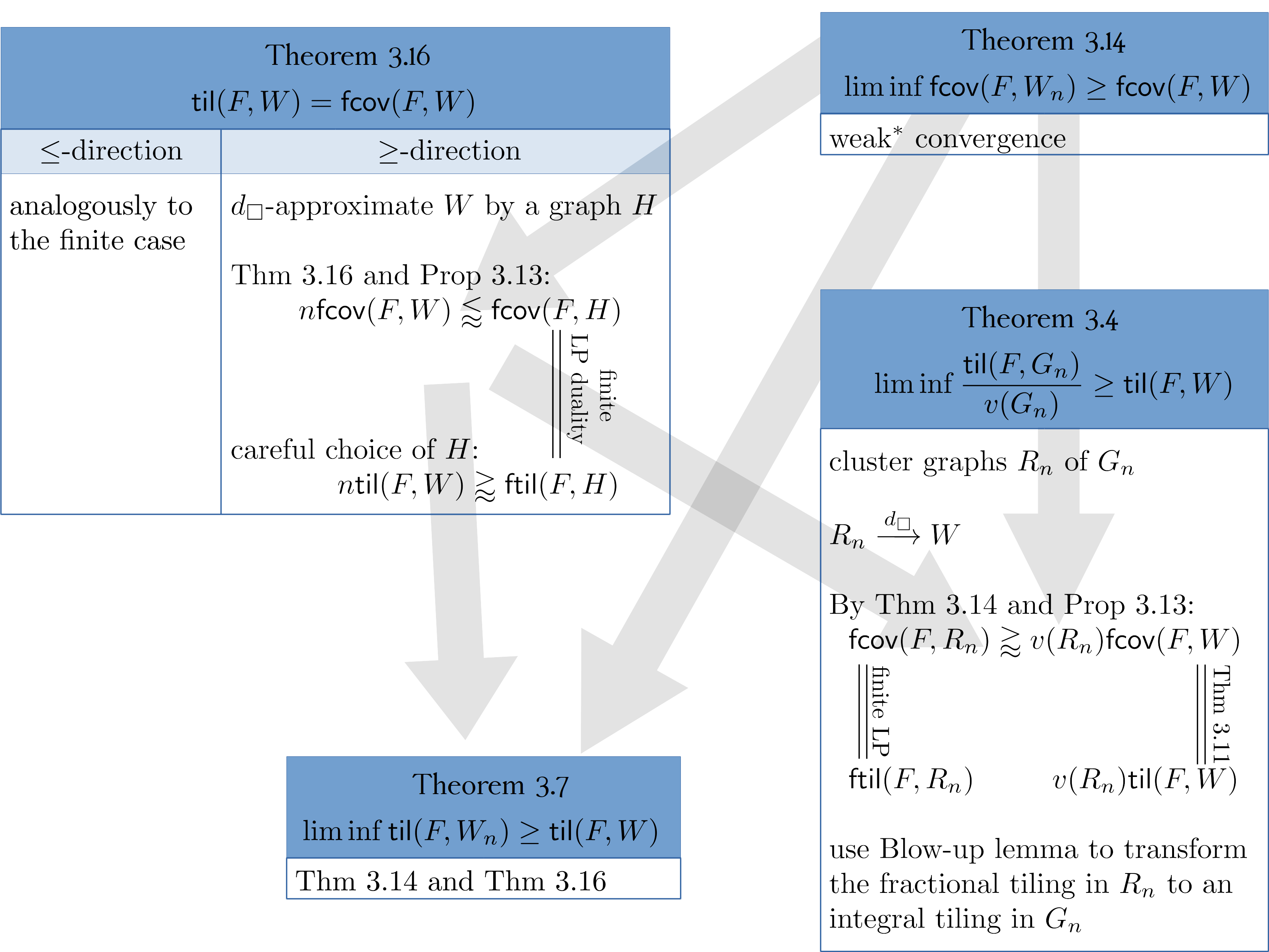}\caption{Main proof steps and dependencies of main results from Section~\ref{sec:StatementOfTheResults}.}
\label{fig:proofStruct}
\end{figure}

\subsection{Proof of Theorem~\ref{thm:limitOfCovers}\label{subsec:ProofTheoremLimitOfCovers}}

Without loss of generality let us assume that $\mathfrak{c}_{n}\WEAKCONV\mathfrak{c}$.
In order to show that $\mathfrak{c}$ is a fractional $F$-cover of
$W$, we need to show that for any measurable subsets $A_{1},\ldots,A_{k}\subset\Omega$
for which there are numbers $\alpha_{1},\ldots,\alpha_{k}$, $\sum\alpha_{i}<1$
with $\mathfrak{c}_{\restriction A_{i}}<\alpha_{i}$, we have $\int_{(x_{1},\ldots,x_{k})\in\prod A_{i}}W^{\otimes F}(x_{1},\ldots,x_{k})\diffr\nu^{k}=0$.
Let $a=1-\sum\alpha_{i}$. Suppose for contradiction that $\int_{(x_{1},\ldots,x_{k})\in\prod A_{i}}W^{\otimes F}(x_{1},\ldots,x_{k})\diffr\nu^{k}>0$.
Let $\xi>0$ be such that the set 
\[
\mathbf{X}=\left\{ \mathbf{x}\in\prod_{i}A_{i}:W^{\otimes F}\left(\mathbf{x}\right)>\xi\right\} 
\]
has positive $\nu^{k}$-measure. Since $k$-dimensional boxes generate
the sigma-algebra on $\Omega^{k}$, for each $d>0$ we can find sets
$B_{1}\subset A_{1},\ldots,B_{k}\subset A_{k}$ of positive measure
such that
\begin{equation}
\nu^{k}\left(\mathbf{X}\cap\prod_{i}B_{i}\right)\ge(1-d)\prod_{i}\nu(B_{i})\;.\label{eq:James}
\end{equation}
Let us fix such sets for $d:=(\nicefrac{a}{12k})^{k}$.

Let us take $N_{1}$ such that for each $n>N_{1}$ we have $\|W-W_{n}\|_{\square}<(\nicefrac{a}{13k})^{k+2}\cdot\prod_{i}\nu(B_{i})\cdot\xi$.
Lemma~\ref{lem:cutdistVSsubgraphCount} tells us that for each collection
$D_{1},\ldots,D_{k}$ of measurable subsets of $\Omega$ we have
\begin{equation}
\left|\int_{(x_{1},\ldots,x_{k})\in\prod D_{i}}W_{n}^{\otimes F}(x_{1},\ldots,x_{k})\diffr\nu^{k}-\int_{(x_{1},\ldots,x_{k})\in\prod D_{i}}W^{\otimes F}(x_{1},\ldots,x_{k})\diffr\nu^{k}\right|<(\nicefrac{a}{13k})^{k}\cdot\prod_{i}\nu(B_{i})\cdot\xi\;.\label{eq:D}
\end{equation}

Since $\mathfrak{c}_{n}\WEAKCONV\mathfrak{c}$, there exists a number
$N_{2}$ such that for each $n>N_{2}$ we have
\begin{equation}
\int_{x\in B_{i}}\mathfrak{c}_{n}(x)\diffr\nu\le\int_{x\in B_{i}}\mathfrak{c}(x)+\frac{a}{3k}\diffr\nu\le\nu(B_{i})\left(\alpha_{i}+\frac{a}{3k}\right)\quad\mbox{for each \ensuremath{i=1,\ldots,k}.}\label{eq:wcconv}
\end{equation}

Now, fix $n>\max(N_{1},N_{2})$. For $i=1,\ldots,k$, define $D_{i}=B_{i}\cap\mathfrak{c}_{n}^{-1}\left([0,\alpha_{i}+\nicefrac{2a}{3k}]\right)$.
Then we have
\begin{eqnarray*}
\int_{x\in B_{i}}\mathfrak{c}_{n}(x)\diffr\nu & = & \int_{x\in D_{i}}\mathfrak{c}_{n}(x)\diffr\nu+\int_{x\in B_{i}\setminus D_{i}}\mathfrak{c}_{n}(x)\diffr\nu\\
 & \ge & \nu(D_{i})\cdot0+\nu(B_{i}\setminus D_{i})\cdot\left(\alpha_{i}+\nicefrac{2a}{3k}\right)\\
 & = & \nu(B_{i})\left(\alpha_{i}+\nicefrac{2a}{3k}-\frac{\nu(D_{i})}{\nu(B_{i})}\cdot\left(\alpha_{i}+\nicefrac{2a}{3k}\right)\right)\;.
\end{eqnarray*}
Plugging this into~(\ref{eq:wcconv}) we get
\[
\nu(B_{i})\left(\alpha_{i}+\nicefrac{2a}{3k}-2\frac{\nu(D_{i})}{\nu(B_{i})}\right)\le\nu(B_{i})\left(\alpha_{i}+\nicefrac{a}{3k}\right)\;,
\]
and consequently,
\begin{equation}
\nu(D_{i})\ge\frac{a}{6k}\cdot\nu(B_{i})\;.\label{eq:Brown}
\end{equation}
Observe also that for each $\mathbf{x}\in\prod D_{i}$ we have $\sum_{i}\mathfrak{c}_{n}(\mathbf{x}_{i})\le\sum_{i}\left(\alpha_{i}+\nicefrac{2a}{3k}\right)=1-\nicefrac{a}{3}<1$.
Since $\mathfrak{c}_{n}$ is a fractional $F$-cover of $W_{n}$,
we conclude that
\begin{equation}
\int_{(x_{1},\cdots,x_{k})\in\prod D_{i}}W_{n}^{\otimes F}(x_{1},\ldots,x_{k})\diffr\nu^{k}=0\;.\label{eq:DD}
\end{equation}
On the other hand,
\begin{eqnarray*}
\int_{(x_{1},\cdots,x_{k})\in\prod D_{i}}W_{n}^{\otimes F}(x_{1},\ldots,x_{k})\diffr\nu^{k} & \ge & \int_{(x_{1},\cdots,x_{k})\in\mathbf{X}\cap\prod D_{i}}\xi\diffr\nu^{k}\\
\JUSTIFY{by\,\eqref{eq:James}} & \ge & \left(\prod_{i}\nu(D_{i})-d\prod_{i}\nu(B_{i})\right)\xi\\
\JUSTIFY{by\,\eqref{eq:Brown}} & \ge & (\nicefrac{a}{12k})^{k}\prod_{i}\nu(B_{i})\cdot\xi\;.
\end{eqnarray*}
This, together with~(\ref{eq:DD}) contradicts~(\ref{eq:D}).

\subsection{Proof of Theorem~\ref{thm:LPdualityGraphons}\label{subsec:ProofTheoremLPDualityGraphons}}

We split the proof into the easy ``$\le$''-part and the difficult
``$\ge$''-part. The former is given in Proposition~\ref{prop:LPdualityEasy}
and its proof is a straightforward modification of the finite version.
The latter is given in Proposition~\ref{prop:LPdualityHard}. In
the proof of Proposition~\ref{prop:LPdualityHard} we first approximate
a graphon by a finite graph on which we use finite LP-duality as a
black-box.

We remark that this approach seems ad-hoc and it is reasonable to
ask for a more conceptual proof. Indeed, as was pointed out to us
by an anonymous referee, there is fairly rich literature on ``algebraic
infinite-dimensional duality'' (see~\cite{MR893179} for a survey).
This theory would certainly not be able to give the results we need.
The referee also pointed out to us ``topological infinite-dimensional
duality'' developed in~\cite{MR1980661} which they suggested might
provide a blackbox proof of Theorem~\ref{thm:LPdualityGraphons}.
We were, however, unable to apply this theory in our setting.
\begin{prop}
\label{prop:LPdualityEasy}Suppose that $W:\Omega^{2}\rightarrow[0,1]$
is a graphon defined on $\Omega$ and $F$ is a graph on the vertex
set $[k]$. Suppose that $\mathfrak{c}:\Omega\rightarrow[0,1]$ is
a fractional $F$-cover of $W$ and that $\mathfrak{t}:\Omega^{k}\rightarrow[0,1]$
is an $F$-tiling in $W$. Then $\|\mathfrak{t}\|\le\|\mathfrak{c}\|$.
\end{prop}

\begin{proof}
We have 
\begin{eqnarray*}
\|\mathfrak{t}\| & =\\
\JUSTIFY{by~\eqref{eq:tilingsupport}} & = & \int_{\left(x_{1},\ldots,x_{k}\right)\in\mathcal{F}_{F}(W)}\mathfrak{t}\left(x_{1},\ldots,x_{k}\right)\diffr\nu^{k}\\
\JUSTIFY{Definition~\ref{def:covergraphon}} & \le & \int_{\left(x_{1},\ldots,x_{k}\right)\in\mathcal{F}_{F}(W)}\mathfrak{t}\left(x_{1},\ldots,x_{k}\right)\left(\sum_{i=1}^{k}\mathfrak{c}\left(x_{i}\right)\right)\diffr\nu^{k}\\
 & = & \int_{x\in\Omega}\left(\mathfrak{c}(x)\sum_{i=1}^{k}\int_{\substack{\left(x_{1},\ldots,,x_{i-1},x_{i+1},\ldots,x_{k}\right)\in\Omega^{k-1}\\
\left(x_{1},\ldots,x_{i-1},x,x_{i+1},\ldots,x_{k}\right)\in\mathcal{F}_{F}(W)
}
}\mathfrak{t}\left(x_{1},\ldots,x_{i-1},x,x_{i+1},\ldots,x_{k}\right)\right)\diffr\nu\\
\JUSTIFY{by~\eqref{eq:IntegralAtMostOne}} & \le & \int_{x\in\Omega}\mathfrak{c}(x)\diffr\nu\;,
\end{eqnarray*}
as was needed.
\end{proof}
\begin{prop}
\label{prop:LPdualityHard}Suppose that $W:\Omega^{2}\rightarrow[0,1]$
is a graphon defined on $\Omega$ and $F$ is a graph on the vertex
set $[k]$. Then for an arbitrary $\epsilon>0$ there exists an $F$-tiling
$\mathfrak{t}:\Omega^{k}\rightarrow[0,1]$ in $W$ with $\|\mathfrak{t}\|\ge\FCOV(F,W)-\epsilon$.
\end{prop}

\begin{proof}
For each $n=1,2,\ldots$, we find a suitable $k_{n}\in\mathbb{N}$,
a partition $\Omega=\Omega_{1}^{(n)}\sqcup\Omega_{2}^{(n)}\sqcup\ldots\sqcup\Omega_{k_{n}}^{(n)}$
into $k_{n}$ sets of measure $\frac{1}{k_{n}}$ each, and a graphon
$W_{n}:\Omega^{2}\rightarrow[0,1]$ such that $W_{n}$ is constant
on each rectangle $\Omega_{i}^{(n)}\times\Omega_{j}^{(n)}$ ($i,j\in[k_{n}]$),
and $\|W-W_{n}\|_{1}<\nicefrac{1}{n}$. This is possible as step functions
with square steps are dense in $\mathcal{L}^{1}(\Omega^{2})$. We
shall now modify the graphons $W_{n}$ in two steps.

First, let $W_{n}':\Omega^{2}\rightarrow[0,1]$ be defined as 
\[
W_{n}'(x,y)=\begin{cases}
W_{n}(x,y) & \mbox{if \ensuremath{W_{n}(x,y)\ge\sqrt{\nicefrac{1}{n}}}}\\
0 & \mbox{otherwise}
\end{cases}\;.
\]
This way, we have that 
\begin{equation}
\|W_{n}-W_{n}'\|_{1}<\sqrt{\nicefrac{1}{n}}\;.\label{eq:WWone}
\end{equation}
For a fixed $n$, we say that a rectangle $\Omega_{i}^{(n)}\times\Omega_{j}^{(n)}$
is \emph{shoddy} if the measure of the set 
\[
\left(\SUPPORT W_{n}'\setminus\SUPPORT W\right)\cap\left(\Omega_{i}^{(n)}\times\Omega_{j}^{(n)}\right)
\]
is at least $\sqrt[4]{\nicefrac{1}{n}}\cdot\frac{1}{k_{n}^{2}}$.
Note that at each point of $\SUPPORT W_{n}'\setminus\SUPPORT W$ the
difference between $W_{n}$ and $W$ is at least $\sqrt{\nicefrac{1}{n}}$.
This gives that $\nu^{2}\left(\SUPPORT W_{n}'\setminus\SUPPORT W\right)\le\frac{\|W_{n}-W\|_{1}}{\sqrt{\nicefrac{1}{n}}}\le\sqrt{\nicefrac{1}{n}}$,
which in turn implies that at most $\sqrt[4]{\nicefrac{1}{n}}\cdot k_{n}^{2}$
rectangles $\Omega_{i}^{(n)}\times\Omega_{j}^{(n)}$ are shoddy. Let
us define 
\[
W_{n}''(x,y)=\begin{cases}
0 & \mbox{\ensuremath{(x,y)} lies in a shoddy rectangle}\\
W'_{n}(x,y) & \mbox{otherwise}
\end{cases}\;.
\]
Observe that shoddy rectangles are symmetric, so $W_{n}''$ is indeed
a graphon. We changed $W''_{n}$ on a set of measure at most $\sqrt[4]{\nicefrac{1}{n}}$
which gives that 
\begin{equation}
\|W_{n}'-W_{n}''\|_{1}\le\sqrt[4]{\nicefrac{1}{n}}\;.\label{eq:WoneWtwo}
\end{equation}

To summarize, we ended up with a graphon $W''_{n}$ which is a step-function
on the rectangles $\Omega_{i}^{(n)}\times\Omega_{j}^{(n)}$ and for
which we have
\[
\|W_{n}''-W\|_{1}\le\|W_{n}-W\|_{1}+\|W_{n}'-W_{n}\|_{1}+\|W''_{n}-W_{n}'\|_{1}\overset{\eqref{eq:WWone},\eqref{eq:WoneWtwo}}{<}\frac{1}{n}+\sqrt{\nicefrac{1}{n}}+\sqrt[4]{\nicefrac{1}{n}}\quad\overset{n\rightarrow\infty}{\;\longrightarrow\;}\quad0\;.
\]
Since the $\mathcal{L}^{1}(\Omega^{2})$-norm is stronger than the
cut-norm, we have that the graphons $(W_{n}'')$ converge to $W$.
In particular, Theorem~\ref{thm:limitOfCovers} tells us that $\liminf_{n}\FCOV(F,W_{n}'')\ge\FCOV(F,W)$.
Let $N_{1}$ be such that for each $n\ge N_{1}$ we have $\FCOV(F,W_{n}'')\ge\FCOV(F,W)-\nicefrac{\epsilon}{2}$.
Let $N_{2}=\lceil\nicefrac{k}{\epsilon}\rceil^{8}$. Now, fix an arbitrary
number $N\ge\max(N_{1},N_{2})$, and consider the graphon $W_{N}''$.
This graphon represents a finite graph on the vertex set $V=\{1,\ldots,k_{N}\}$.
We call this graph $H$. We forget the edge-weights, that is we put
an edge $ij$ in $H$ whenever $W_{N}''$ is positive on $\Omega_{i}^{(N)}\times\Omega_{j}^{(N)}$.
By Proposition~\ref{prop:FCOVgraphFCOV}, $\FCOV(F,H)=k_{N}\cdot\FCOV(F,W_{N}'')$.

The LP-duality of finite graphs tells us that we can find a fractional
$F$-tiling $\mathfrak{t}_{H}:V^{k}\rightarrow[0,1]$ on $H$ with
$\sum\mathfrak{t}_{H}(v_{1},\ldots,v_{k})=\FCOV(F,H)$. It is easy
to transform $\mathfrak{t}_{H}$ to a fractional $F$-tiling $\mathfrak{t}:\Omega^{k}\rightarrow[0,\infty)$
on $W_{N}''$ by defining $\mathfrak{t}$ to be the constant $k_{N}\cdot\mathfrak{t}_{H}(v_{1},\ldots,v_{k})$
on the entire rectangle $\prod_{i}\Omega_{v_{i}}^{(N)}$ (for each
$v_{1},v_{2},\ldots,v_{k}\in[k_{N}])$. Then $\mathfrak{t}$ has weight
$\frac{1}{k_{N}}\sum\mathfrak{t}_{H}(v_{1},\ldots,v_{k})=\FCOV(F,W_{N}'')$.
The function $\mathfrak{t}$ is not necessarily a fractional $F$-tiling
on $W$, as condition~(\ref{eq:tilingsupport}) may be violated due
to values at points $P=\SUPPORT\left(\left(W_{N}''\right)^{\otimes F}\right)\setminus\SUPPORT\left(W^{\otimes F}\right)$.
We have 
\begin{equation}
P\subset\bigcup_{i<j,ij\in E(F)}\pi_{ij}^{-1}\left(\SUPPORT W_{n}''\setminus\SUPPORT W\right)\;.\label{eq:Pij}
\end{equation}
Let $\mathfrak{t}'$ be zero on $P$, and equal to $\mathfrak{t}$
elsewhere. Then $\mathfrak{t}'$ is a fractional $F$-tiling on $W$. 

Since $\SUPPORT W_{n}''\setminus\SUPPORT W$ is disjoint from points
of shoddy rectangles, we have that for each choice of $v_{1},\ldots,v_{k}$,
the product set $\prod_{i=1}^{k}\Omega_{v_{i}}^{(N)}$ satisfies
\begin{eqnarray*}
\nu^{k}\left(\prod_{i=1}^{k}\Omega_{v_{i}}^{(N)}\cap P\right) & \overset{\eqref{eq:Pij}}{\le} & \sum_{\ell<j,ij\in E(F)}\nu^{k}\left(\prod_{i=1}^{k}\Omega_{v_{i}}^{(N)}\cap\pi_{\ell j}^{-1}\left(\SUPPORT W_{N}''\setminus\SUPPORT W\right)\right)\\
 & \le & \nu^{k}\left(\prod_{i=1}^{k}\Omega_{v_{i}}^{(N)}\right)\binom{k}{2}\sqrt[4]{\nicefrac{1}{N}}\;.
\end{eqnarray*}
Using that $\mathfrak{t}$ is constant on $\prod_{i=1}^{k}\Omega_{v_{i}}^{(N)}$,
we get that 
\[
\int_{(x_{1},\ldots,x_{k})\in\prod_{i=1}^{k}\Omega_{v_{i}}^{(N)}}\mathfrak{t}'(x_{1},\ldots,x_{k})\diffr\nu^{k}\ge\left(1-\tbinom{k}{2}\sqrt[4]{\nicefrac{1}{N}}\right)\int_{(x_{1},\ldots,x_{k})\in\prod_{i=1}^{k}\Omega_{v_{i}}}\mathfrak{t}(x_{1},\ldots,x_{k})\diffr\nu^{k}\;.
\]
Summing over all product sets $\prod_{i=1}^{k}\Omega_{v_{i}}^{(N)}$
we get that $\mathfrak{t}'$ is a fractional $F$-tiling on $W$ with
\begin{align*}
\int_{(x_{1},\ldots,x_{k})\in\Omega^{k}}\mathfrak{t}'(x_{1},\ldots,x_{k})\diffr\nu^{k} & \ge\left(1-\tbinom{k}{2}\sqrt[4]{\nicefrac{1}{N}}\right)\int_{(x_{1},\ldots,x_{k})\in\Omega^{k}}\mathfrak{t}(x_{1},\ldots,x_{k})\diffr\nu^{k}\\
 & \ge\int_{(x_{1},\ldots,x_{k})\in\Omega^{k}}\mathfrak{t}(x_{1},\ldots,x_{k})\diffr\nu^{k}-\frac{\epsilon}{2}\\
 & =\FCOV(F,W_{N}'')-\frac{\epsilon}{2}\ge\FCOV(F,W)-\epsilon\;,
\end{align*}
as was needed.
\end{proof}

\subsection{Proof of Theorem~\ref{thm:lowersemicont}\label{subsec:ProofTheoremLowesemicontinuous}}

By Theorem~\ref{thm:LPdualityGraphons} it suffices to prove that
$\liminf_{n}\FCOV(F,W_{n})\ge\FCOV(F,W)$. This, however, is the subject
of Corollary~\ref{cor:liminfCovers}.

\subsection{Proof of Theorem~\ref{thm:lowersemicontgraphs}\label{subsec:ProofTheoremLowesemicontinuousGraphs}}

For simplicity, let us assume that $v(G_{n})=n$. Suppose that $\epsilon>0$
is arbitrary. We want to show that for $n$ sufficiently large, $\TIL(F,G_{n})\ge(\TIL(F,W)-\epsilon)n=(\FCOV(F,W)-\epsilon)n$
(where the last equality uses Theorem~\ref{thm:LPdualityGraphons}).
Let $\delta>0$ be such that whenever $U$ is a graphon of cut-distance
at most $\delta$ from $W$ we have that 
\begin{equation}
\FCOV(F,U)\ge\FCOV(F,W)-\nicefrac{\epsilon}{3}\;.\label{eq:FCovbig}
\end{equation}
Such a number $\delta$ exists by Corollary~\ref{cor:liminfCovers}.
Let $d=\nicefrac{\delta}{10}$. Let $\epsilon_{1}$ be given by Lemma~\ref{lem:weakblowup}
for the error parameter $\gamma=\nicefrac{\epsilon}{10}$ and density
$\nicefrac{d}{2}$. Let $\epsilon_{0}$ be given by Lemma~\ref{lem:randomsubpair}
for the input parameter $\beta=0.1$ and error parameter $\epsilon_{1}$.
Set $\epsilon_{2}=\min(\epsilon_{0},\nicefrac{\delta}{10})$. Lemma~\ref{lem:RegularityLemma}
with error parameter $\epsilon_{2}$ gives us a bound $L$ on the
complexity of partitions.

Suppose $n$ is large enough, so that we have
\begin{equation}
\DIST_{\square}(W_{G_{n}},W)<\nicefrac{\delta}{2}\;.\label{eq:distWgnW}
\end{equation}

Consider a cluster graph $R_{n}$ corresponding to a $(\epsilon_{2},d)$-regularization
of the graph $G_{n}$. Let $\mathcal{V}=(V_{1},\ldots,V_{\ell})$
be the corresponding partition of $V(G_{n})$ into clusters, with
$\ell\le L$. Let $m=|V_{1}|$. By~(\ref{eq:distGraphClustergraph})
we have $\DIST_{\square}(W_{G_{n}},W_{R_{n}})<\nicefrac{\delta}{2}$.
Combining with~(\ref{eq:distWgnW}), triangle inequality gives $\DIST_{\square}(W,W_{R_{n}})<\delta$.
In particular, (\ref{eq:FCovbig}) can be applied to $W_{R_{n}}$.
Proposition~\ref{prop:FCOVgraphFCOV} then gives that $\FCOV(F,R_{n})=\ell\cdot\FCOV(F,W_{R_{n}})\ge\ell\left(\FCOV(F,W)-\nicefrac{\epsilon}{3}\right)$.
We ignore the edge weights in $R_{n}$. We apply the (standard) LP
duality on $R_{n}$. We get a fractional $F$-tiling $\mathfrak{t}_{R_{n}}:[\ell]^{k}\rightarrow[0,1]$.
We now delete from $\mathfrak{t}_{R_{n}}$ very small weights. Formally,
for a $k$-tuple $\mathbf{v}\in[\ell]^{k}$ we define $\mathfrak{t}'_{R_{n}}(\mathbf{v})$
to be~0 if $\mathfrak{t}{}_{R_{n}}(\mathbf{v})<\nicefrac{\epsilon}{\left(10\ell^{k-1}\right)}$
and to be $\mathfrak{t}_{R_{n}}(\mathbf{v})$ otherwise. Observe that
$\|\mathfrak{t}'{}_{R_{n}}\|\ge\|\mathfrak{t}{}_{R_{n}}\|-\nicefrac{\epsilon\ell}{10}$.
For each copy $\tilde{F}$ of the graph $F$ inside $R_{n}$ which
is in the support of $\mathfrak{t}'{}_{R_{n}}$, consider an arbitrary
integer $m_{\tilde{F}}$ in the range between $\left[\left(1-\nicefrac{\epsilon}{10}\right)\cdot\mathfrak{t}'{}_{R_{n}}(\tilde{F})\cdot m,\mathfrak{t}'{}_{R_{n}}(\tilde{F})\cdot m\right]$.
Since we can assume that $n$ is sufficiently large, for the gap between
the endpoints of this range we have
\[
\frac{\epsilon}{10}\cdot\mathfrak{t}'{}_{R_{n}}(\tilde{F})\cdot m\ge\frac{\epsilon}{10}\cdot\frac{\epsilon}{10\ell^{k-1}}\cdot\frac{(1-\epsilon_{2})n}{\ell}\ge1\;.
\]
In particular, we can indeed find an integer $m_{\tilde{F}}$ in the
specified range.

Now, for each cluster $V_{i}$, we construct a partition of $V_{i}$
into sets $V_{i,0}\sqcup\bigsqcup V_{i,\tilde{F}}$ where $\tilde{F}$
runs over all copies of $F$ in $R_{n}$ that touch the vertex $i$.
In this partition, the size of each set $V_{i,\tilde{F}}$ is $m_{\tilde{F}}$,
and $V_{i,0}$ is the remainder. Since $\mathfrak{t}'_{R_{n}}$ is
an $F$-tiling $_{\tilde{F}:\tilde{F}\ni i}m_{\tilde{F}}\le\sum_{\tilde{F}:\tilde{F}\ni i}\mathfrak{t}'{}_{R_{n}}(\tilde{F})\cdot m\le m$.
That means that at least one such partition exists. Among all such
partitions $V_{i,0}\sqcup\bigsqcup V_{i,\tilde{F}}$, consider a random
one. 

Suppose that this way we randomly partitioned all the clusters $V_{1},\ldots,V_{\ell}$.
For a given copy $\tilde{F}$ of $F$ in $R_{n}$ in the support of
$\mathfrak{t}'_{R_{n}}$, we set up a random variable $X_{\tilde{F}}$
as follows. If for each edge $ij\in\tilde{F}$ the pair $(V_{i,\tilde{F}},V_{j,\tilde{F}})$
is $\epsilon_{1}$-regular, we set $X_{\tilde{F}}=(1-\nicefrac{\epsilon}{10})m_{\tilde{F}}$.
If there exists an edge $ij\in\tilde{F}$ for which the pair $(V_{i,\tilde{F}},V_{j,\tilde{F}})$
fails to be $\epsilon_{1}$-regular, we set $X_{\tilde{F}}=0$. Lemma~\ref{lem:weakblowup}
gives that one can always find an $F$-tiling of size $X_{\tilde{F}}$
in $G_{n}\left[\bigcup_{i\in V(\tilde{F})}V_{i,\tilde{F}}\right]$.

Lemma~\ref{lem:randomsubpair} implies that $\EXP\left[X_{\tilde{F}}\right]\ge\exp\left(-2e(F)0.1^{m_{\tilde{F}}}\right)(1-\nicefrac{\epsilon}{10})m_{\tilde{F}}\ge(1-\nicefrac{\epsilon}{5})m_{\tilde{F}}$.
By linearity of expectation, 
\begin{eqnarray*}
\EXP\left[\sum_{\tilde{F}}X_{\tilde{F}}\right] & \ge & (1-\nicefrac{\epsilon}{5})\sum_{\tilde{F}}\left(1-\nicefrac{\epsilon}{10}\right)\cdot\mathfrak{t}'{}_{R_{n}}(\tilde{F})\cdot m\ge(1-\nicefrac{\epsilon}{3})\|\mathfrak{t}'_{R_{n}}\|m\\
 & \ge & (1-\nicefrac{\epsilon}{3})\|\mathfrak{t}{}_{R_{n}}\|m-\nicefrac{\epsilon n}{10}\;.
\end{eqnarray*}
So, let us fix one collection of partitions such that $\sum_{\tilde{F}}X_{\tilde{F}}\ge(1-\nicefrac{\epsilon}{3})\|\mathfrak{t}{}_{R_{n}}\|m-\nicefrac{\epsilon n}{10}$.
Then by the above there is an $F$-tiling in $G_{n}$ of size $\sum_{\tilde{F}}X_{\tilde{F}}$.
Using the fact that $\FCOV(F,R_{n})\ge\ell\left(\FCOV(F,W)-\nicefrac{\epsilon}{3}\right)$,
we obtain that $\TIL(F,G_{n})\ge(\FCOV(F,W)-\epsilon)n$, as needed.

\subsection{Proof of Proposition~\ref{prop:uppercont}\label{subsec:ProofPropUpperCont}}

Write $k=v(F)$. Given $\eta>0$, let the number $\beta$ be given
by the Removal lemma~\ref{lem:removalpartiteNullify} for the graph
$F$ and the parameters $\alpha=\left(\nicefrac{\eta}{2}\right)^{k+2}$
and $d=\eta/2$. Set $\delta=\beta(\eta/2)^{4k^{2}}/k^{2}$. 

Suppose now that $U$ is $\delta$-close to $W$ in the cut-norm.
Let us fix an optimal fractional $F$-cover~$\mathfrak{c}$ of $W$
(recall Theorem~\ref{thm:LPdualityGraphons} and (\ref{eq:fcovattained})).
Let $\widehat{U}$ and $\mathfrak{c}^{'}$ be defined by
\begin{eqnarray*}
\widehat{U}(x,y) & = & \begin{cases}
U(x,y) & \mbox{if \ensuremath{U(x,y)>\eta/2}}\\
0 & \mbox{if \ensuremath{U(x,y)\le\eta}}/2
\end{cases}\;,\\
\mathfrak{c}'(x) & = & \max\left(\mathfrak{c}(x)+\eta,1\right)\;.
\end{eqnarray*}
Obviously, $\|U-\widehat{U}\|_{1}\le\eta/2$ and $\|\mathfrak{c}'\|\le\|\mathfrak{c}\|+\eta$.
Let $\ell=\lceil\eta^{-1}\rceil$. Let us consider the sets $A_{i}$,
$i=1,\ldots,\ell$ defined by $A_{i}=\mathfrak{c}^{-1}[0,\nicefrac{i}{\ell})$.
Let $X\subset\Omega^{k}$ be the set of points $(x_{1},\ldots,x_{k})$
where $\widehat{U}{}^{\otimes F}(x_{1},\ldots,x_{k})>0$ but $\sum_{i}\mathfrak{c}'(x_{i})<1$.
Let $\mathcal{I}\subset\mathbb{N}^{k}$ be the set of all $k$-tuples
of integers whose sum is $\ell$. We have $|\mathcal{I}|\le\left(\frac{2}{\eta}\right)^{k}$.
Observe that 
\begin{equation}
X\subset\bigcup_{\mathbf{i}\in\mathcal{I}}A_{\mathbf{i}_{1}}\times A_{\mathbf{i}_{2}}\times\cdots\times A_{\mathbf{i}_{k}}\;.\label{eq:Xcontained}
\end{equation}
The next claim is crucial.
\begin{claim*}
For each $k$-tuple $\mathbf{i}\in\mathcal{I}$, we have for $\mathbf{A}=A_{\mathbf{i}_{1}}\times A_{\mathbf{i}_{2}}\times\cdots\times A_{\mathbf{i}_{k}}$
that $\nu^{k}(\mathbf{A}\cap X)\le\beta$.
\end{claim*}
\begin{proof}
Suppose for a contradiction that $\nu^{k}(\mathbf{A}\cap X)>\beta$.
Therefore, we have
\begin{equation}
\int_{(x_{1},\ldots,x_{k})\in\mathbf{A}}U{}^{\otimes F}(x_{1},\ldots,x_{k})\diffr\nu^{k}\ge\int_{(x_{1},\ldots,x_{k})\in\mathbf{A}}\widehat{U}{}^{\otimes F}(x_{1},\ldots,x_{k})\diffr\nu^{k}\ge\left(\nicefrac{\eta}{2}\right)^{e(F)}\cdot\beta>k^{2}\delta\;.\label{eq:IntegralUdash-1}
\end{equation}

On the other hand, for each $k$-tuple $\mathbf{x}\in\mathbf{A}$
we have $\sum_{j}\mathfrak{c}(\mathbf{x}_{j})<1$. Since $\mathfrak{c}$
is a fractional $F$-cover of $W$, we conclude that $W^{F}(\mathbf{x})=0$.
Therefore,
\begin{equation}
\int_{(x_{1},\ldots,x_{k})\in\mathbf{A}}W{}^{\otimes F}(x_{1},\ldots,x_{k})\diffr\nu^{k}=0\;.\label{eq:IntegralW-1}
\end{equation}

Combining~\ref{eq:IntegralUdash-1} and~\ref{eq:IntegralW-1} with
Lemma~\ref{lem:cutdistVSsubgraphCount} we get that $\|U-W\|_{\square}>\delta$,
which is a contradiction.
\end{proof}
We shall use tools introduced in Section~\ref{subsec:PartiteGraphons}.
For each $\mathbf{i}\in\mathcal{I}$, let $U_{\mathbf{i}}$ be the
$\left(A_{\mathbf{i}_{1}},\ldots,A_{k}\right)$-partite version of
$\widehat{U}$ defined on an auxiliary measure space $\Omega_{\mathbf{i}}=A_{\mathbf{i}_{1}}\amalg\ldots\amalg A_{\mathbf{i}_{k}}$
with a measure $\nu_{\Omega_{\mathbf{i}}}$. The above Claim tells
us that $\int_{(x_{1},\ldots,x_{k})\in\left(\Omega_{\mathbf{i}}\right)^{k}}\widehat{U_{\mathbf{i}}}^{\otimes F}\diffr\left(\nu_{\Omega_{\mathbf{i}}}\right)\le\beta$.
The Removal Lemma (Lemma~\ref{lem:removalpartiteNullify}) tells
us that there exists a set $S_{\mathbf{i}}\subset\Omega_{\mathbf{i}}^{2}$
of $\Omega_{\mathbf{i}}^{2}$-measure at most $\alpha$ such that
nullifying $\widehat{U_{\mathbf{i}}}$ on $S_{\mathbf{i}}$ yields
an $F$-free graphon. Let $B_{\mathbf{i}}\subset\Omega^{2}$ be the
folded version of $S_{\mathbf{i}}$. We have that $\nu^{2}(B_{\mathbf{i}})\le\alpha$
by~(\ref{eq:projected}). Let us now nullify $\widehat{U}$ on $B=\bigcup_{\mathbf{i}\in\mathcal{I}}B_{\mathbf{i}}$.
The set $B$ has $\nu^{2}$-measure at most $|\mathcal{I}|\cdot\alpha\le\frac{2}{\eta^{k}}\cdot\alpha\le\nicefrac{\eta}{2}$,
and thus the resulting graphon $U^{*}$ satisfies $\|U-U^{*}\|_{1}\le\eta$.
The nullification together with~(\ref{eq:Xcontained}) tells us that
$\mathfrak{c}'$ is a fractional $F$-cover of $U^{*}$. This finishes
the proof.

\subsection{Proof of Theorem~\ref{thm:robusttilingnumber} \label{subsec:ProofRobustTiling}}

Suppose that a sequence of graphons $\left(W_{n}\right)_{n}$ converges
to a graphon $W$ in the cut-norm. We shall prove the statement in
two steps:
\begin{eqnarray}
\limsup_{n}\TIL_{\epsilon}(F,W_{n}) & \ge & \TIL_{\epsilon}(F,W)\;,\mbox{and}\label{eq:ROBUSTge}\\
\limsup_{n}\TIL_{\epsilon}(F,W_{n}) & \le & \TIL_{\epsilon}(F,W)\;.\label{eq:ROBUSTle}
\end{eqnarray}
(Note that by passing to a subsequence, we could turn the limit superior
into a limit.)

First, let us prove~(\ref{eq:ROBUSTge}). For each $n$, suppose
that $U_{n}\le W_{n}$ is an arbitrary graphon of $\mathcal{L}^{1}(\Omega^{2})$-norm
at most $\epsilon$. Since the space of graphons is sequentially compact,
let us consider the limit $U$ of a suitable subsequence $\left(U_{n_{i}}\right)_{i}$.
Since $U_{n_{i}}\le W_{n_{i}}$ for each $i$, we also have $U\le W$.
Furthermore, since the $\mathcal{L}^{1}(\Omega^{2})$-norm is continuous
with respect to the cut-distance, we have that the $\mathcal{L}^{1}(\Omega^{2})$-norm
of $U$ is at most $\epsilon$. In particular, $U$ appears in the
infimum in~(\ref{eq:defTILrobust}) for the graphon $W$. We then
have 
\[
\limsup_{n}\TIL(F,W_{n}-U_{n})\ge\liminf_{i}\TIL(F,W_{n_{i}}-U_{n_{i}})\;\overset{T\ref{thm:lowersemicont}}{\ge}\;\TIL(F,W-U)\ge\TIL_{\epsilon}(F,W)\;,
\]
as was needed for~(\ref{eq:ROBUSTge}).

For~(\ref{eq:ROBUSTle}) we shall need that the function $\TIL_{\delta}(F,W)$,
considered as a function in $\delta\in(0,1)$ is left-continuous. 
\begin{claim*}
Suppose that $\delta\in(0,1)$. Then in the setting above, for each
$\eta>0$ and each graphon $U\le W$ with $\|U\|_{1}=\delta$ there
exists a graphon $U'\le U$ with $\|U'\|_{1}<\delta$ such that $\TIL(F,W-U)\le\TIL(F,W-U')+\eta$.
\end{claim*}
\begin{proof}[Proof of Claim]
 Clearly, there exists a set $A\subset\Omega$ of measure at most
$\eta$ such that $U_{\restriction A\times\Omega}$ is positive on
a set of positive measure. Let us nullify $U$ on $(A\times\Omega)\cup(\Omega\times A)$.
For the resulting graphon $U'$ we have $\|U'\|_{1}<\delta$. 

Suppose now that $\mathfrak{t}$ is an arbitrary $F$-tiling in $W-U'$.
By nullifying $\mathfrak{t}$ on those $v(F)$-tuples whose at least
one coordinate lies in $A$, we obtain an $F$-tiling $\mathfrak{t}^{*}$
in $W-U$. By~(\ref{eq:IntegralAtMostOne}), we have $\|\mathfrak{t}\|\le\|\mathfrak{t}^{*}\|+\eta$.
We conclude that $\TIL(F,W-U)\le\TIL(F,W-U')+\eta$.
\end{proof}
By the above claim, it suffices to prove the following weaker form
of~(\ref{eq:ROBUSTle}):
\begin{equation}
\limsup_{n}\TIL_{\epsilon}(F,W_{n})\le\TIL_{\epsilon-\nicefrac{1}{\ell}}(F,W)+\nicefrac{1}{\ell}\;,\mbox{for each \ensuremath{\ell\in\mathbb{N}}.}\label{eq:ROBUSTreduced}
\end{equation}
So, suppose that $\ell\in\mathbb{N}$ is arbitrary. Let $U\le W$
be an arbitrary graphon of $\mathcal{L}^{1}(\Omega^{2})$-norm at
most $\epsilon-\nicefrac{1}{\ell}$. Let $\mathfrak{c}:\Omega\rightarrow[0,1]$
be an arbitrary fractional $F$-cover of $W-U$ of size $\TIL(F,W-U)$.
Such a cover exists by Theorem~\ref{thm:LPdualityGraphons}. Let
us consider the sets $A_{i}$, $i=1,\ldots,\ell-1$ defined by $A_{i}=\mathfrak{c}^{-1}[\nicefrac{\left(i-1\right)}{\ell},\nicefrac{i}{\ell})$
and $A_{\ell}=\mathfrak{c}^{-1}[\nicefrac{\left(\ell-1\right)}{\ell},1]$.
We therefore have another fractional $F$-cover $\mathfrak{c}^{*}=\sum_{i}\frac{i}{\ell}\mathbf{1}_{A_{i}}$
of size at most $\|\mathfrak{c}\|+\nicefrac{1}{\ell}$. 

For each $n$, let us take a graphon $U_{n}\le W_{n}$ so that the
sequence $\left(U_{n}\right)_{n}$ converges to $U$ in the cut-norm.
Such a sequence $\left(U_{n}\right)_{n}$ exists by Lemma~\ref{lem:subgraphonscoverge}.
Since the $\mathcal{L}^{1}(\Omega^{2})$-norm of $U$ is at most $\epsilon-\nicefrac{1}{\ell}$
and since the $\mathcal{L}^{1}(\Omega^{2})$ topology is stronger
than the cut-norm topology, we can additionally assume that the $\mathcal{L}^{1}(\Omega^{2})$-norm
of each $U_{n}$ is at most $\epsilon-\nicefrac{1}{\ell}$. By the
same argument as in the proof of Proposition~\ref{prop:uppercont},
for large enough $n$, the density of copies of $F$ in $W_{n}-U_{n}$
not covered by the $\mathfrak{c}^{*}$ is $o_{n}(1)$. Using the Removal
lemma, for large enough $n$, we can find graphons $U_{n}'\le W_{n}-U_{n}$
so that $U_{n}'$ has $\mathcal{L}^{1}(\Omega^{2})$-norm at most
$\nicefrac{1}{2\ell}$ and that $W_{n}-(U_{n}+U_{n}')$ has zero density
of copies of $F$ in not covered by the $\mathfrak{c}^{*}$. Since
$U_{n}+U_{n}'$ has $\mathcal{L}^{1}(\Omega^{2})$-norm at most $\epsilon$,
we get that 
\[
\TIL_{\epsilon}(F,W_{n})\le\|\mathfrak{c}^{*}\|\le\|\mathfrak{c}\|+\nicefrac{1}{\ell}\le\TIL(F,W-U)+\nicefrac{1}{\ell}\;,
\]
as was needed for~(\ref{eq:ROBUSTreduced}).

\section{Tilings in inhomogeneous random graphs\label{sec:RandomGraphs}}

In this section we give a simple application of our theory. It concerns
the random graph model $\mathbb{G}(n,W)$ which was introduced by
Lovász and Szegedy in~\cite{Lovasz2006}. Let us briefly recall the
model. If $W:\Omega^{2}\rightarrow[0,1]$ is a graphon, then to sample
a graph from the distribution $\mathbb{G}(n,W)$, $G\sim\mathbb{G}(n,W)$,
we take deterministically $V(G)=[n]$. Further, we sample points $x_{1},\ldots,x_{n}\in\Omega$
independently at random according to the law $\nu$. To define the
edges of $G$, we include each pair $ij$ as an edge in $G$ with
probability $W(x_{i},x_{j})$, independently of the other choices.
When $W$ is constant $p$, $\mathbb{G}(n,W)$ is the usual Erd\H{o}s\textendash Rényi
random graph $\mathbb{G}(n,p)$. See~\cite[Chapter 10]{Lovasz2012}
for more properties of the model $\mathbb{G}(n,W)$. 

We prove that the ratios $\frac{\TIL\left(F,\mathbb{G}(n,W)\right)}{n}$
and $\frac{\FTIL\left(F,\mathbb{G}(n,W)\right)}{n}$ converge to $\TIL(F,W)$
asymptotically almost surely. The proof of this statement is a short
application of Theorem~\ref{thm:LPdualityGraphons}.
\begin{thm}
\label{thm:InhomogeneousRandomGraph}Suppose that $W:\Omega^{2}\rightarrow[0,1]$
is a graphon. Then the values $\frac{\TIL\left(F,\mathbb{G}(n,W)\right)}{n}$
and $\frac{\FTIL\left(F,\mathbb{G}(n,W)\right)}{n}$ converge in probability
to the constant $\TIL(F,W)$.
\end{thm}

\begin{proof}
\noindent It is well-known (see e.g.~\cite[Lemma 10.16]{Lovasz2012})
that the sequence of graphs $\left(\mathbb{G}(n,W)\right)_{n}$ converges
to $W$ in the cut-distance almost surely. Thus by Theorem~\ref{thm:lowersemicontgraphs},
$\frac{\TIL\left(F,\mathbb{G}(n,W)\right)}{n}$ is asymptotically
almost surely at least $\TIL(F,W)-o(1)$ . The analogous statement
for $\frac{\FTIL\left(F,\mathbb{G}(n,W)\right)}{n}$ holds, as for
any graphs $F$ and $G$ we have $\TIL(F,G)\le\FTIL\left(F,G\right)$.

\medskip{}

\noindent Now, we pick an arbitrary $\ell\in\mathbb{N}$ and we show
that asymptotically almost surely, $\frac{\FTIL\left(F,\mathbb{G}(n,W)\right)}{n}$
is at most $\TIL(F,W)+\nicefrac{3}{\ell}$. Let us apply Theorem~\ref{thm:LPdualityGraphons}
and fix a fractional $F$-cover $\mathfrak{c}:\Omega\rightarrow[0,1]$
of size $\TIL(F,W)$. Let us round $\mathfrak{c}$ up to the closest
multiple of $\nicefrac{1}{2\ell}$. This way, the size of the modified
fractional $F$-cover $\mathfrak{c}'$ increased by at most $\nicefrac{1}{2\ell}$.
For $i=0,\ldots,2\ell$, define $\Omega_{i}$ to be the preimage of
$\nicefrac{i}{2\ell}$ under $\mathfrak{c}'$. Since $\mathfrak{c}'$
is a fractional $F$-cover, we have that for each $k$-tuple $i_{1},i_{2},\ldots,i_{k}$
with $\sum_{j}i_{j}<2\ell$ that 
\[
\int_{\left(x_{1},\ldots,x_{k}\right)\in\prod_{\ell=1}^{k}\Omega_{i_{\ell}}}W^{\otimes F}(x_{1},\ldots,x_{k})\diffr\nu^{k}=0\;.
\]
Since the integrand is non-negative, we get $W^{\otimes F}(\mathbf{x})=0$
for almost every $k$-tuple $\mathbf{x}=(x_{1},\ldots,x_{k})$ as
above. Then, for such a tuple $\mathbf{x}$, there exist two indices
$p_{\mathbf{x}},q_{\mathbf{x}}$, $1\le p_{\mathbf{x}}<q_{\mathbf{x}}\le k$,
$p_{\mathbf{x}}q_{\mathbf{x}}\in E(F)$ such that 
\begin{equation}
\mbox{\ensuremath{W(x_{p_{\mathbf{x}}},x_{q_{\mathbf{x}}})=0}.}\label{eq:WzeroRect}
\end{equation}

Let us now sample the random graph $G\sim\mathbb{G}(n,W)$. Let $y_{1},y_{2},\ldots,y_{n}\in\Omega$
be the random points that represent the $n$ vertices $1,2,\ldots,n$
of $\mathbb{G}(n,W)$. By the Law of Large Numbers, asymptotically
almost surely we have for each $j=0,1,\ldots,2\ell$, 
\begin{equation}
\left|\left\{ y_{1},\ldots,y_{n}\right\} \cap\Omega_{j}\right|\le\left(\nu(\Omega_{j})+\nicefrac{1}{\ell^{2}}\right)n\;.\label{eq:aec}
\end{equation}
Define a function $\mathfrak{b}:V(G)\rightarrow[0,1]$ by mapping
a vertex $i$ to $\mathfrak{c}'(y_{i})$. We then have
\begin{eqnarray*}
\|\mathfrak{b}\| & = & \sum_{i=1}^{n}\mathfrak{c}'(y_{i})=\sum_{j=0}^{2\ell}\left|\left\{ y_{1},\ldots,y_{n}\right\} \cap\Omega_{j}\right|\cdot\frac{j}{2\ell}\\
 & \overset{\eqref{eq:aec}}{\le} & \sum_{j=0}^{2\ell}\left(\nu(\Omega_{j})+\nicefrac{1}{\ell^{2}}\right)n\cdot\frac{j}{2\ell}\le(\|\mathfrak{c'}\|+\nicefrac{2}{\ell})n\le(\|\mathfrak{c}\|+\nicefrac{3}{\ell})n\;.
\end{eqnarray*}
We claim that $\mathfrak{b}$ is a fractional $F$-cover of $G$ with
probability~1. Indeed, let $m_{1},m_{2},\ldots,m_{k}\in V(G)$ be
arbitrary with 
\begin{equation}
\sum_{j=1}^{k}\mathfrak{b}(m_{j})<1\;.\label{eq:bm}
\end{equation}
It is our task to show that there exist $pq\in E(F)$ such that $m_{p}m_{q}\not\in E(G)$.
To this end, we observe that (\ref{eq:bm}) translates as $y_{m_{1}}\in\Omega_{i_{1}},y_{m_{2}}\in\Omega_{i_{2}},\ldots,y_{m_{k}}\in\Omega_{i_{k}}$
for some $k$-tuple $i_{1},i_{2},\ldots,i_{k}$ with $\sum_{j}i_{j}<2\ell$.
Thus, (\ref{eq:WzeroRect}) applies for some numbers $p,q\in[k]$.
But then indeed the edge $m_{p}m_{q}$ was included with probability
$W(y_{m_{p}},y_{m_{q}})=0$ in $\mathbb{G}(n,W)$, as was needed.

Since $\ell$ was arbitrary, we obtain that $\frac{\FTIL\left(F,\mathbb{G}(n,W)\right)}{n}$
is asymptotically almost surely at most $\TIL(F,W)+o(1)$. The analogous
statement for $\frac{\TIL\left(F,\mathbb{G}(n,W)\right)}{n}$ follows
from the fact that for any graphs $F$ and $G$ we have $\TIL(F,G)\ge\FTIL\left(F,G\right)$.
\end{proof}
It is plausible that Theorem~\ref{thm:InhomogeneousRandomGraph}
can be be extended even to sparse inhomogeneous random graphs $\mathbb{G}(n,p_{n}W)$,
where $\left(p_{n}\right)_{n}$ is a sequence of positive reals tending
to zero. 

As an immediate corollary of Theorem~\ref{thm:InhomogeneousRandomGraph}
and Theorem~\ref{thm:lowersemicontgraphs} we obtain the following
corollary.
\begin{cor}
\label{cor:limesinferior}Suppose that $F$ is an arbitrary graph
and $W$ is an arbitrary graphon. Then
\[
\liminf\frac{\TIL(F,G_{n})}{v(G_{n})}=\TIL(F,W)\;,
\]
 where the limit inferior ranges over all graph sequences $(G_{n})$
that converge to $W$.
\end{cor}

\section{A concluding remark: Analytic approach to linear programming in the
setting of graphons}

Our main results are proven by first approximating a graphon by a
step-function and then applying discrete tools, such as linear programming.
Arguments that would work directly in the setting of functional analysis
would certainly be much shorter and more elegant. As was pointed out
by the referee, some such tools are already available in~\cite{MR893179}
and~\cite{MR1980661}, but we were unable to make use of them in
the current context. In a recent paper~\cite{HlaRoch:IndepCliCol},
Hladký and Rocha investigate concepts of the chromatic number and
the clique number of graphons. Unlike in finite graphs, the LP duality
fails for the fractional versions of these parameters.

\section{Acknowledgments}

JH would like thank Dan Král and András Máthé for useful discussions
that preceded this project. He would also like to thank Martin Doležal
for the discussions they have had regarding functional analysis. We
thank two anonymous referees for their comments.

Part of this paper was written while JH was participating in the program
\emph{Measured group theory} at The Erwin Schrödinger International
Institute for Mathematics and Physics.

\medskip{}

The contents of this publication reflects only the authors' views
and not necessarily the views of the European Commission of the European
Union. This publication reflects only its authors' view; the European
Research Council Executive Agency is not responsible for any use that
may be made of the information it contains.

\bibliographystyle{plain}
\bibliography{bibl}

\end{document}